\documentclass[12pt]{article}
\usepackage{enumerate}

\usepackage{amsthm}
\usepackage{bbm}
\usepackage{amsfonts}
\usepackage{mathrsfs}
\usepackage{amsmath}
\usepackage{fancyhdr}
\usepackage{hyperref}
\usepackage{color} 
\usepackage{mathtools}

\usepackage{float}

\usepackage{enumitem}
\usepackage{graphicx}
\usepackage{titlesec}
\usepackage{bigints}
\usepackage{relsize,exscale}
\usepackage{romannum}
\usepackage{multirow}
\usepackage{subcaption}
\usepackage[nottoc,numbib]{tocbibind}
\usepackage{tikz}
\usetikzlibrary{decorations.markings}
\usetikzlibrary{arrows}

\usetikzlibrary{trees}

\usepackage{arxiv}
\usepackage{hyperref} 

\theoremstyle{plain}

\newtheorem{Lemma}{Lemma}   
\newtheorem{prop}{Proposition}   

\theoremstyle{definition}

\theoremstyle{newremark}
\newtheorem{remark}{Remark}

\newtheorem{conjecture}{Conjecture}

\newcommand{\s}{\sigma}

\newcommand{\x}{\mathbf{x}}

\renewcommand{\a}{\alpha}
\newcommand{\e}{\epsilon}
\renewcommand{\d}{\delta}

\newcommand{\slim}{\sum\limits_{k=1}^\infty}
\newcommand{\slimj}{\sum\limits_{j=i}^\infty}
\newcommand{\slimo}{\sum\limits_{k=0}^\infty}

\fancyheadoffset{0in}

\title{Linear-Quadratic Stochastic Differential Games on  Directed Chain Networks}
\fancypagestyle{plain}{%
  \fancyhead{} 
}

\author{
\and  Yichen Feng\thanks{Department of Statistics and Applied Probability, South Hall, University of California, Santa Barbara, CA 93106, USA (E-mail: \href{mailto:feng@pstat.ucsb.edu}{feng@pstat.ucsb.edu}).} 
  \and Jean-Pierre Fouque\thanks{Department of Statistics and Applied Probability, South Hall, University of California, Santa Barbara, CA 93106, USA (E-mail: \href{mailto:fouque@pstat.ucsb.edu}{fouque@pstat.ucsb.edu}). Work supported by NSF grant DMS-1814091.} 
 \and Tomoyuki Ichiba\thanks{Department of Statistics and Applied Probability, South Hall, University of California, Santa Barbara, CA 93106, USA (E-mail: \href{mailto:ichiba@pstat.ucsb.edu}{ichiba@pstat.ucsb.edu}). Work supported by NSF grant DMS-1615229 and DMS-2008427.} 
  }
\date{\vspace{-5ex}}

\begin{document}

\maketitle

\pagenumbering{arabic}

\begin{abstract}
 We study  linear-quadratic stochastic differential games on  directed chains inspired by the directed chain stochastic differential equations introduced by Detering, Fouque \& Ichiba \cite{Nils-JP-Ichiba2018DirectedChain}. We solve explicitly for Nash equilibria with a finite number of players and we study  more general finite-player games with a mixture of both directed chain interaction and mean field interaction. We investigate and compare the corresponding games in the limit when the number of players tends to infinity. 
 The limit is characterized by Catalan functions and the dynamics under equilibrium is an infinite-dimensional Gaussian process described by 
 a Catalan Markov chain, with or without the presence of mean field interaction.
\end{abstract}

\noindent{{{\it Key Words and Phrases:} Linear-quadratic stochastic games, directed chain network, Nash equilibrium, Catalan functions, Catalan Markov chain, mean field games.}


\section{Introduction}
Stochastic differential games on networks is a broad area. There are two extreme situations. On one hand, we can consider a fully connected network with interaction of  mean-field  type. When the number of players goes to infinity, this kind of game can be approximated by a mean field game. The mean field convergence problem has been discussed widely, for instance in Lacker \cite{ConvNashtoMFGlimit}. Other networks and games have been proposed and studied. For example, Delarue \cite{delarue:hal-01457409} investigates an example of a game with a large number of players in mean-field interaction when the graph connection between them is of Erdos-R\'enyi type, and Lacker,  Ramanan \& Wu \cite{LargeNetworkof_InteractingDiffusions} study the limit of an interacting diffusive particle system on a large sparse interaction graph with finite average degree.
On the other hand, we can consider a very structured network such as a one-dimensional directed chain which has been studied in Detering, Fouque \& Ichiba \cite{Nils-JP-Ichiba2018DirectedChain} without the game aspect. It is a complete opposite to mean field games since, on a directed chain network, each player interacts with its neighbor in a given direction. In this paper, we introduce a game aspect of the directed chain and identify  Nash equilibria. We also consider the limit when the number of players goes to infinity.

Interestingly, the equilibrium dynamics on the network discussed in this paper turns out to be different from the dynamics suggested in \cite{Nils-JP-Ichiba2018DirectedChain}, in particular, with long time variance behavior. The equilibrium dynamics for the infinite-player game is described by a Catalan Markov chain introduced in this paper.

Our first goal is to consider a  game on a directed chain network and to find its Nash equilibrium. We focus on open-loop Nash equilibria. We want to understand how the structure of the network affects this Nash equilibrium. We propose three directed chain networks shown in Figures \ref{fig:chain} and \ref{fig:ring}. Starting from a finite directed chain, we  also discuss a periodic directed chain in a ring structure and we  compare with the game on a infinite directed chain network. 

The paper is organized as follows. In section \ref{section 2}, we propose a finite-player game model on a directed chain and construct an open-loop Nash equilibrium. We discuss  general boundary conditions as well as two special cases to illustrate that the boundary condition  actually affects weakly the Nash equilibrium. We also observe that for this type of games open-loop and closed-loop Nash equilibria coincide. Section \ref{section 3} is devoted to the analysis of an infinite-player stochastic differential game on a directed chain. We try to find an open-loop Nash equilibrium and get a similar Riccati system to that of the finite-player game. The solutions are called Catalan functions and we use them to build a Catalan Markov chain, discussed in section \ref{section 4}. We find that its long-time asymptotic variance and covariance are finite.  In sections \ref{section 5} and \ref{section 6}, we discuss the finite-player and infinite-player games for a mixed system including both a directed chain interaction and a mean-field interaction. We can adjust the model to be a purely mean field game (studied in \cite{CarmonaFouqueSunSystemicRisk}), or a purely directed chain game, or a mixture of the two by introducing a tuning parameter $u\in [0,1]$. We repeat the same steps as sections \ref{section 2},  \ref{section 3}, and \ref{section 4} to find the Nash equilibria and we construct a generalized Catalan Markov chain describing the two effects. We find that the long-time asymptotic variance of the process with the purely directed chain interaction is finite, which is different from the case with mean-field interaction as shown in Table 1 in \cite{Nils-JP-Ichiba2018DirectedChain}. In section \ref{section 7},  we propose a finite-player periodic directed chain game and we construct an open-loop Nash equilibrium. We conjecture that its infinite-player limit is the same as the one found for other boundary condition. This conjecture is supported by numerical results. In Section \ref{section-tree-model}, we extend our results to tree structures with fixed finite number of descendants. Section \ref{section 8} gives a conclusion and  open problems. Appendix \ref{Appendix} includes some technical proofs and discussions.

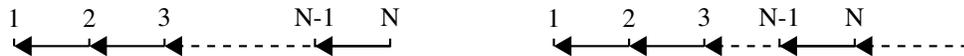
\begin{figure}[!h]
	\centering
	\begin{minipage}[t]{4cm}
		\centering
		\begin{tikzpicture}[scale=1,mydashed/.style={dashed,dash phase=3pt},>={triangle 60}]
\draw[thick] (0,0.1) -- (0,0) node[label=above:1] {};
\draw[thick] (1,0.1) -- (1,0) node[label=above:2] {};
\draw[thick] (2,0.1) -- (2,0) node[label=above:3] {};
\draw[thick] (4,0.1) -- (4,0)--(5,0) -- (5,0.1);
\draw[thick,<-] (0,0) -- (1,0);
\draw[thick,<-] (1,0) -- (2,0);
\draw[thick,mydashed,<-] (2,0) -- (4,0) node[label=above:{N-1}] {};
\draw[thick,<-] (4,0) -- (5,0) node[label=above:N] {};
\end{tikzpicture}
	\end{minipage}
	\hspace{3cm}
	\begin{minipage}[t]{4cm}
		\centering
		\begin{tikzpicture}[mydashed/.style={dashed,dash phase=3pt},>={triangle 60}]
\draw[thick] (0,0.1) -- (0,0) node[label=above:1] {};
\draw[thick] (1,0.1) -- (1,0) node[label=above:2] {};
\draw[thick] (2,0.1) -- (2,0) node[label=above:3] {};
\draw[thick] (3,0.1) -- (3,0)--(4,0) -- (4,0.1);
\draw[thick,<-] (0,0) -- (1,0);
\draw[thick,<-] (1,0) -- (2,0);
\draw[thick,mydashed,<-] (2,0) -- (3,0) node[label=above:{N-1}] {};
\draw[thick,<-] (3,0) -- (4,0) node[label=above:N] {};
\draw[thick,mydashed,<-] (4,0) -- (5.5,0);
\end{tikzpicture}
	\end{minipage}
	\caption{Finite Directed Chain (Left) and Infinite Directed Chain (Right)}
	\label{fig:chain}
\end{figure}

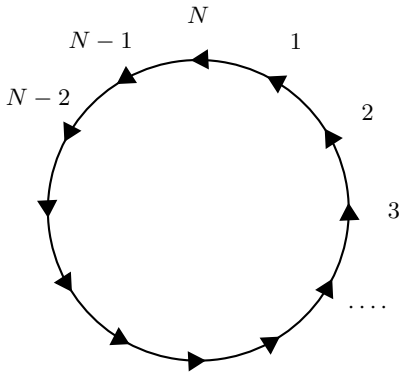
\begin{figure}[!h]
\centering
 \begin{tikzpicture}[scale=2,cap=round,>={triangle 60}]
\draw[thick, decoration={markings, mark=between positions 0.008 and 1 step {1/12} with {\arrow{>}}},               
        postaction={decorate}
        ]
        (0,0) circle (1cm);
          \foreach \x in {0,30,...,360} {
                \filldraw[black] (\x:1cm) circle(0.3pt);
        }
    
        \foreach \x/\xtext in {
            0/3,
            30/2,
            60/1,
            90/N, 
            145/{N-2},
            120/{N-1},
            330/{\cdots\cdot}}
                \draw (\x:1.30cm) node[fill=white] {\small $\xtext$};

\end{tikzpicture}
\caption{Periodic Directed Chain}
\label{fig:ring}
\end{figure}

\section{N-Player Directed Chain Game}\label{section 2}
\subsection{Setup and Assumptions}
We consider a stochastic game in continuous time, involving $N$ players indexed from $1$ to $N$. Each player $i$ is controlling its own, real-valued private state $X_t^i$  by taking a real-valued action $\alpha_t^i$ at time $t\in [0,T]$. The dynamics of the states of the $N$ individual players are given by $N$ stochastic differential equations of the form:
 \begin{equation}\label{eq:1}
    dX_t^i=\a_t^i {\mathrm d}t+\s {\mathrm d}W_t^i,\quad  i=1,\cdots,N,
\end{equation}
where $0\leq t\leq T$ and $ (W_t^i)_{0\leq t\leq T},\, i = 1, \cdots, N$ are independent standard Brownian motions. Here and throughout the paper, the argument in the superscript represents index or label but not the power. For simplicity, we assume that the diffusion is one-dimensional and the diffusion coefficients are constant and identical denoted by $\sigma>0$. The drift coefficients $\alpha^i$'s are adapted to the filtration of the Brownian motions and satisfy $\mathbbm{E}[\int_0^T |\alpha_t^i|^2 dt]<\infty$ for $i = 1, \ldots , N$. 
The system starts at time $t = 0$ from $i.i.d.$ square-integrable random variables $X_0^i = \xi_i$, independent of the Brownian motions and, without loss of generality,  we assume ${\mathbbm E}(\xi_i) = 0$ for $i = 1, \ldots , N$. 

In this model, among the first $N-1$ players, each player $i$ chooses its own strategy $\alpha^i$, in order to minimize its objective function given by:
\begin{equation} \label{eq: objfunc1}
1\leq   i\leq N-1:\quad\quad J^i(\a^1,\cdots,\a^N)=\mathbbm{E} \left\{\int_0^T \left(\frac{1}{2}(\a^i_t)^2+\frac{\e}{2}(X_t^{i+1}-X_t^{i})^2\right){\mathrm d}t+\frac{c}{2}(X_T^{i+1}-X_T^{i})^2\right\},
\end{equation}
for some constants $\e>0$ and $c\geq 0$. The running cost and the terminal cost functions are defined by 
\begin{equation} \label{eq: fg1}
f^i(x,\a^i)=\frac{1}{2}(\a^i)^2+\frac{\e}{2}(x^{i+1}-x^{i})^2, \quad \text{ and } \quad g^i(x)=\frac{c}{2}(x^{i+1}-x^{i})^2, 
\end{equation} 
respectively for $x = (x^{1}, \ldots , x^{N}) \in \mathbb R^{N}$ and $\alpha^{i} \in \mathbb R$, $i = 1, \ldots , N$. This is a \emph{Linear-Quadratic} differential game on a directed chain network, since the state $X^{i}$ of each player $i$ interacts only with $X^{i+1}$ through the quadratic cost functions for $i =1, \ldots , N-1$. The system is completed by describing the behavior of player $N$ which will be done in the following section, when we discuss the boundary condition of the system.

\subsection{Open-Loop Nash Equilibrium} 

In this section, we search for an open-loop Nash equilibrium of the system of $N$ players  
among the admissible strategies $\{\alpha_t^i,i=1,\cdots,N, t \in [0, T] \}$ and we study the effect of boundary conditions induced by the behavior of player $N$. We will discuss a general boundary condition for the game first in Section \ref{sec: 2.2.1} and then show two particular choices in Section \ref{BC-attractedto0} 
and \ref{BC-notattractedto0}. 
We construct the equilibrium by the Pontryagin stochastic maximum principle.

\subsubsection{General Boundary Condition} \label{sec: 2.2.1}

We consider a setup with a general boundary condition for the directed chain where the last player $N$ does not depend on the other players.
The expected cost functional for player $N$ is defined by:
\begin{equation} \label{eq: JNalphaN} 
J^N(\alpha^N)=\mathbbm{E} \left\{\int_0^T \left(\frac{1}{2}(\alpha^N_t)^2+q_2(X_t^N)\right){\mathrm d}t+Q_2(X_T^N)\right\}, 
\end{equation}
\begin{equation}  \label{eq: JNalphaN2}
\text{ where }  \quad q_2(x)=\frac{a_1}{2}(x-m)^2+a_2, \quad \text{ and } \quad  Q_2(x)=\frac{c_1}{2}(x-m)^2+c_2, \quad x \in \mathbb R 
\end{equation} 
are non-degenerate convex quadratic functions in $x$, where $a_1,a_2,m,c_1,c_2$ are some constants with $a_1>0$ and $c_1>0$. The running cost function is defined by $f^N(x, \alpha^N)=\frac{1}{2}(\alpha^N)^2+q_2(x)$ and the terminal cost function is defined by $g^N(x)=Q_2(x)$. This can be seen as a control problem for the player $N$ and we assume its state is attracted to some constant level $m \in \mathbb R$.

The Hamiltonian for player $i\leq N-1$ is given by:
\[
    H^i(x^1,\cdots,x^N,y^{i,1},\cdots,y^{i,N},\alpha^1,\cdots,\alpha^N)=\sum\limits_ {k=1}^N \alpha^k y^{i,k}+\frac{1}{2}(\alpha^i)^2+\frac{\epsilon}{2}(x^{i+1}-x^i)^2,
\]
while the Hamiltonian for player $N$ is:
\[
H^N(x^1,\cdots,x^N,y^{i,1},\cdots,y^{i,N},\alpha^1,\cdots,\alpha^N)=\sum\limits_ {k=1}^N \alpha^k y^{i,k}+\frac{1}{2}(\alpha^N)^2+\frac{a_1}{2}(x^N-m)^2+a_2
\]
for $x^{k}, y^{i,k}, \alpha^{k} \in \mathbb R$, $i,k =1, \ldots , N$. 
For $i = 1, \ldots , N$ the value of $\alpha^i$ minimizing the Hamiltonian $H^{i}(\cdot)$ with respect to $\alpha^i$, when all the other variables including $\alpha^j$ for $j\neq i$ are fixed, is given by the first order condition 
\[
    \partial_{\alpha^i}H^i=y^{i,i}+\alpha^i=0 \quad \text{leading to the choice:} \quad  \Hat{\alpha}^i=-y^{i,i}.
\]

The adjoint processes $Y_t^i=(Y_t^{i,j};j=1,\cdots,N)$ and $Z_t^i=(Z_t^{i,j,k};j=1,\cdots,N,k=1,\cdots,N)$ for $i=1,\cdots,N$ are defined as the solutions of the system of  backward stochastic differential equations (BSDEs): for $j = 1, \ldots , N$
\begin{equation} 
\begin{array}{ll}
  i\leq N-1:&  \left\{
  \begin{array}{ll}
    {\mathrm d}Y_t^{i,j}
    &=-\partial_{x^j}H^i(X_t,Y_t^i,\alpha_t){\mathrm d}t+\sum\limits_ {k=1}^N Z_t^{i,j,k}{\mathrm d}W_t^k\\
    &=-\epsilon (X_t^{i+1}-X_t^i)(\delta_{i+1,j}-\delta_{i,j}){\mathrm d}t+\sum\limits_ {k=1}^N Z_t^{i,j,k}{\mathrm d}W_t^k, \quad 0 \le t \le T, \\
    Y_T^{i,j}&=\partial_{x^j}g_i(X_T)=c(X_T^{i+1}-X_T^i)(\delta_{i+1,j}-\delta_{i,j});
      \end{array}
    \right.\\
    
  i=N: & \left\{
  \begin{array}{ll}
    {\mathrm d}Y_t^{N,j} &=-a_1 (X_t^N-m) \delta_{N,j}{\mathrm d}t+\sum\limits_ {k=1}^N Z_t^{N,j,k}{\mathrm d}W_t^k, \quad  0 \le t \le T , \\
    Y_T^{N,j}&=c_1 (X_T^N-m) \delta_{N,j}.
    \end{array}
    \right.
  \end{array}
\end{equation}
for $0 \le t \le T$. Particularly, for $j=i$, $j = i+1$, it becomes:
\begin{equation} \label{BSDE-general}
\left\{
  \begin{array}{lll}
    & {\mathrm d} Y_t^{i,i} =\epsilon (X_t^{i+1}-X_t^i){\mathrm d}t+\sum\limits_ {k=1}^N Z_t^{i,i,k}{\mathrm d}W_t^k,\quad &Y_T^{i,i}=-c(X_T^{i+1}-X_T^i) ,\quad i\leq N-1, \\
    & {\mathrm d} Y_{t}^{i,i+1} \, =\,  - \epsilon (X_t^{i+1}-X_t^i){\mathrm d}t+\sum\limits_ {k=1}^N Z_t^{i,i+1,k}{\mathrm d}W_t^k,\quad &Y_T^{i,i+1}= c(X_T^{i+1}-X_T^i) ,\quad i\leq N-1, \\
    & {\mathrm d}Y_t^{N,N} =-a_1 (X_t^N-m){\mathrm d}t+\sum\limits_ {k=1}^N Z_t^{N,N,k}{\mathrm d}W_t^k,\quad &Y_T^{N,N}=c_1(X_T^{N}-m).
 \end{array}
\right.
\end{equation}
Thus, because of $\, Y_{T}^{i,i} \, =\,  - Y_{T}^{i,i+1}\,$ and of the form of dynamics, it is reduced to  
\begin{equation}
Y_{t}^{i, i} = - Y_{t}^{i,i+1} , \quad Z_{t}^{i,i,k} = - Z_{t}^{i,i+1,k} \, 
\end{equation}
for $ i \le N-1, \, k \le N, \, 0 \le t \le T  $. For $j \neq i , i+1$, $i \le N-1$, it becomes: 
${\mathrm d}Y_{t}^{i,j} \, =\,  \sum_{k=1}^{N} Z_{t}^{i,j,k} {\mathrm d} W_{t}^{k} $, $Y_{T}^{i,j} \, =\,  0 $, 
and hence, the solution is 
\begin{equation} \label{eq: YijZijk = 0}
\, Y_{t}^{i,j} \equiv 0 \,, \quad \, Z_{t}^{i,j,k} \equiv 0 , \quad \, 0 \le t \le T\,. 
\end{equation}

Considering the BSDE system (\ref{BSDE-general}) and its terminal condition, we  make the ansatz: 
\begin{equation}\label{ansatz-general}
    Y_t^{i,i}=\sum\limits_{j=i}^{N-1} \phi_t^{N,i,j}X_t^j+ \underbrace{(\phi_t^{N,i,N}X_t^N+\psi_t^{N,i})}_\text{affine in $X^N$, depending on B.C.}=\sum\limits_{j=i}^N \phi_t^{N,i,j}X_t^j+\psi_t^{N,i},
\end{equation}
for some deterministic scalar functions $\phi_t$ (depending on $N$) satisfying the terminal conditions: for $1\leq i\leq N-1$, $\phi_T^{N,i,i}=c,\phi_T^{N,i,i+1}=-c,\phi_T^{N,i,j}=0 $ for $ j\geq i+2$, $\psi_T^{N,i}=0$; and $\phi_T^{N,N,N}=c_1$, $\psi_T^{N,N}=-c_1m$.
With this ansatz, the optimal strategy $ \hat{\alpha}_{\cdot} $ and the controlled forward equation for $X_{\cdot}$ in (\ref{eq:1}) become 
\begin{equation} \label{eq: OL-Nash1}
    \left\{
  \begin{array}{ll}
    &\hat{\alpha}_t^i=-Y_t^{i,i}=-\big(\sum\limits_{j=i}^N \phi_t^{N,i,j}X_t^j+\psi_t^{N,i}\big),\\
    & {\mathrm d}X_t^j=-\big(\sum\limits_{k=j}^N \phi_t^{N,j,k}X_t^k+\psi_t^{N,j}\big) {\mathrm d}t+\sigma {\mathrm d}W_t^j.
  \end{array}
\right.
\end{equation}
Differentiating the ansatz (\ref{ansatz-general}) and substituting (\ref{eq: OL-Nash1}) leads to:
\begin{equation}\label{ito-general}
  \begin{array}{ll}
    {\mathrm d}Y_t^{i,i}&=\sum\limits_{j=i}^N[X_t^j\Dot{\phi}_t^{N,i,j} {\mathrm d}t+\phi_t^{N,i,j}dX_t^j]+\dot{\psi}_t^{N,i} {\mathrm d}t\\
   &=\big\{ \sum\limits_{k=i}^N \big(\Dot{\phi}_t^{N,i,k}- \sum\limits_{j=i}^k \phi_t^{N,i,j}\phi_t^{N,j,k} \big) X_t^k+\big[ \dot{\psi}_t^{N,i} -\sum\limits_{j=i}^N \psi_t^{N,j}\phi_t^{N,i,j} \big] \big\}{\mathrm d}t
   +\sigma\sum\limits_{k=i}^N \phi_t^{N,i,k}{\mathrm d}W_t^k.
  \end{array}
\end{equation}
Here $\dot{\phi}_{t}$ represents the time derivative of $\phi_{t}$. Comparing the martingale parts and drifts of two It\^o's decompositions  (\ref{BSDE-general}) and (\ref{ito-general}) of $Y_t^{i,i}$, 
the martingale terms give the deterministic (and therefore adapted) processes 
$Z_t^{i,i,k}$:
\begin{equation}\label{martingale-general}
    Z_t^{i,i,k}=0 \quad \text{ for } \quad k<i, \quad \text{ and }\ \quad Z_t^{i,i,k}=\sigma\phi_t^{N,i,k} \quad \text{ for } \quad k\geq i.
\end{equation}

Moreover, the drift terms show that the functions $\phi_t^{N,\cdot,\cdot}$ and $\psi_t^{N,\cdot}$ must satisfy the system of Riccati equations :\\
for $i\leq N-1$,
\begin{equation}\label{finite-chain-riccati}
\begin{array}{rll}
     \Dot{\phi}_t^{N,i,i}&=\phi_t^{N,i,i}\cdot \phi_t^{N,i,i}-\epsilon, &\phi_T^{N,i,i}=c, \\
      \Dot{\phi}_t^{N,i,i+1}&=\phi_t^{N,i,i}\cdot \phi_t^{N,i,i+1}+\phi_t^{N,i,i+1}\cdot \phi_t^{N,i+1,i+1}+\epsilon, &\phi_T^{N,i,i+1}=-c,\\
     &\vdots &\\
     \Dot{\phi}_t^{N,i,\ell}
     
     &=\phi_t^{N,i,i}\cdot \phi_t^{N,i,\ell}+\phi_t^{N,i,i+1}\cdot\phi_t^{N,i+1,\ell} \\
     & \qquad {} +\cdots+\phi_t^{N,i,\ell-1}\cdot\phi_t^{N,\ell-1,\ell}+\phi_t^{N,i,\ell}\cdot\phi_t^{N,\ell,\ell}, & \phi_T^{N,i,\ell}=0,\\
     &\vdots&\\
      \Dot{\phi}_t^{N,i,N-1}
     &=\phi_t^{N,i,i}\cdot \phi_t^{N,i,N-1}+\cdots+\phi_t^{N,i,N-1}\cdot\phi_t^{N,N-1,N-1},& \phi_T^{N,i,N-1}=0,\\
     \Dot{\phi}_t^{N,i,N}
     &=\phi_t^{N,i,i}\phi_t^{N,i,N}+\cdots+\phi_t^{N,i,N-1}\phi_t^{N,N-1,N}+\phi_t^{N,i,N}\phi_t^{N,N,N},& \phi_T^{N,i,N}=0;
\end{array}
\end{equation}
for $i=N$,
\[
\begin{array}{rll}
\Dot{\phi}_t^{N,N,N}&=\phi_t^{N,N,N}\cdot \phi_t^{N,N,N}-a_1, & \phi_T^{N,N,N}=c_1; 
\end{array}
\]
and $\psi_{\cdot}^{N,j}$, $j \le N$ are determined by 
\begin{equation} \label{finite-chain-riccati2}
\left\{
\begin{array}{lll}
\dot{\psi}_t^{N,i} &=\sum\limits_{j=i}^N \psi_t^{N,j}\phi_t^{N,i,j}, & \psi_T^{N,i}=0,\\
&\vdots\\
\dot{\psi}_t^{N,N-1} &= \psi_t^{N,N-1}\phi_t^{N,N-1,N-1}+\psi_t^{N,N}\phi_t^{N,N-1,N}, & \psi_T^{N,N-1}=0,\\
\dot{\psi}_t^{N,N} &= \psi_t^{N,N}\phi_t^{N,N,N}+a_1 m, & \psi_T^{N,N}=-c_1m.
\end{array}
\right.
\end{equation}
From the equations above, 
the functions $\phi_t^{N,i,i}$ for all $i=1,\cdots,N-1$ are identical; the functions $\phi_t^{N,i,i+1}$ for all $i=1,\cdots,N-2$ are identical ;$\cdots$; and the functions $\phi_t^{N,i,N-2}=\phi_t^{N,i+1,N-1}$. The functions $\phi_t^{N,i,N}$ for all $i$ depend on $\phi_t^{N,N,N}$ of the last player which is determined by the boundary condition. However,
the functions $\phi_t^{N,i,i},\cdots,\phi_t^{N,i,N-1}$ are independent of $\phi_t^{N,i,N}$ and the boundary condition. The functions $\psi^{N,\cdot}$ depend on the $\phi$ functions and have no effect on $\phi^{N,i,j}$ ($j<N$) as well. 

In conclusion, these $\phi^{N,i,j}\,(j<N)$ functions are solvable, identical and independent of the boundary condition as long as the boundary condition defines the last player as a self-controlled problem. The preceding argument is summarized as the following proposition. 

\begin{prop} An open-loop Nash equilibrium for the linear quadratic stochastic game with cost functionals (\ref{eq: objfunc1})-(\ref{eq: fg1}) for the first $N-1$ players and (\ref{eq: JNalphaN})-(\ref{eq: JNalphaN2}) for the $N$th player is given by (\ref{eq: OL-Nash1}), where $\phi^{N,i,j}_{\cdot}$ and $\psi^{N,j}$ are uniquely determined by the system  (\ref{finite-chain-riccati})-(\ref{finite-chain-riccati2}) of Riccati equations.  
\end{prop}

As the number of players goes to infinity, we can get rid of the boundary condition and get a sequence of functions $\{\phi_t^j,j=1,2,\cdots\}$, defined by $
\phi_t^0=\phi_t^{N,i,i}
$, 
$\phi_t^1=\phi_t^{N,i,i+1}$, $\cdots$, 
$
\phi_t^{j}=\phi_t^{N,i,i+j}
$ for large $N$ and so on. 
It indicates that the Nash equilibrium converges to a limit independent of the boundary condition. Therefore, it is natural to study a similar game with infinite players and we conjecture that the limit of the Nash equilibrium of the finite-player game gives us the Nash equilibrium of the infinite-player game. And the sequence of functions $\{\phi_t^j,i\in \mathbb{N}\}$ is the solution to the Riccati equation system of the infinite-player game. This will be discussed in Section \ref{section 3}. Next, two particular examples are discussed to better illustrate the effect of the special boundary condition.

\subsubsection{Boundary Condition 1: $X^N $ is attracted to 0}\label{BC-attractedto0}
Here, we  discuss the case when $X^N$ is attracted to $0$ which is also the common mean $\mathbbm E [ \xi_{i}] \, =\,  0 $ of the initial condition. It is equivalent to the general boundary condition (\ref{eq: JNalphaN})-(\ref{eq: JNalphaN2}) with $m=0$. Without loss of generality, we  can take constants: $a_1=\epsilon$, $c_1=c$ and $a_2=c_2=0$. Then the cost functional for player $N$ is given by: 
\[
J^N(\alpha^N)=\mathbbm{E} \left\{\int_0^T \left(\frac{1}{2}(\alpha^N_t)^2+\frac{\epsilon}{2}(X_t^N)^2\right) {\mathrm d}t+\frac{c}{2}(X_T^N)^2\right\}.
\]
The running cost function is defined by $f^N(x, \alpha^N)=\frac{1}{2}(\alpha^N)^2+\frac{\epsilon}{2}x^2$ and the terminal cost function is defined by $g^N(x)=\frac{c}{2}c^2$. Then, $X^N$ is independent of the other players and is the solution of a self-controlled problem. 
We then make the same ansatz as (\ref{ansatz-general}) with $\psi_t^{N,i}=0$ for all $i$, $0 \le t \le T$. As a result, the martingale terms give the same processes 
$Z_t^{i,i,k}$ as (\ref{martingale-general}).
And from the drift terms, we obtain the system of Riccati equations:\\
for $i\leq N-1$, $0 \le t \le T$,
\[
\begin{array}{rll}
     \Dot{\phi}_t^{N,i,i}&=\phi_t^{N,i,i}\cdot \phi_t^{N,i,i}-\epsilon, &\phi_T^{N,i,i}=c, \\
     \Dot{\phi}_t^{N,i,i+1}&=\phi_t^{N,i,i}\cdot \phi_t^{N,i,i+1}+\phi_t^{N,i,i+1}\cdot \phi_t^{N,i+1,i+1}+\epsilon, &\phi_T^{N,i,i+1}=-c,\\
     &\vdots &\\
     \Dot{\phi}_t^{N,i,l}
     &=\phi_t^{N,i,i}\cdot \phi_t^{N,i,l}+\phi_t^{N,i,i+1}\cdot\phi_t^{N,i+1,l}+\cdots+\phi_t^{N,i,l-1}\cdot\phi_t^{N,l-1,l}+\phi_t^{N,i,l}\cdot\phi_t^{N,l,l}, & \phi_T^{N,i,l}=0,\\
    & \vdots&\\
     \Dot{\phi}_t^{N,i,N}
     &=\phi_t^{N,i,i}\phi_t^{N,i,N}+\phi_t^{N,i,i+1}\phi_t^{N,i+1,N}+\cdots+\phi_t^{N,i,N-1}\phi_t^{N,N-1,N}+\phi_t^{N,i,N}\phi_t^{N,N,N},& \phi_T^{N,i,N}=0;
\end{array}
\]
for $i=N$, $0 \le t \le T$,
\[
\Dot{\phi}_t^{N,N,N}=\phi_t^{N,N,N}\cdot \phi_t^{N,N,N}-\epsilon, \quad \phi_T^{N,N,N}=c, 
\]
From above, we have the same conclusion: the functions $\phi_t^{N,i,i+k}=\phi_t^{N,j,j+k} $ for all $i,j\geq 1,k\geq 1$ and $i+k<N,j+k<N$; and functions $\phi_t^{N,i,j} $($j<N$) are independent of the boundary condition. 

\begin{remark}\label{rem: dependence j-i}
Notice that in this case $\phi_t^{N,N,N}$ has the same solution as $\phi_t^{N,i,i}$ ($i<N$). Thus, in the ansatz (\ref{ansatz-general}), we can actually assume the solution 
$\phi^{N,i,j}_{\cdot}$ depends only on the difference $j-i$ for $j \ge i$. 
\end{remark} 


\subsubsection{Boundary Condition 2: $\alpha^N=0$} \label{BC-notattractedto0}
We  study the case when there is no control for the last player $X^N$, i.e. the dynamics of the state is given by: 
\[
    {\mathrm d}X_t^N=\sigma {\mathrm d}W_t^N, \quad 0 \le t \le T \,;\quad\quad X_0^N=\xi_N,\quad \mathbbm E(\xi_N)=0.
\]
Player $i$ chooses the strategy $\alpha_t^i$ ($i<N$) to minimize $J^i$ given in (\ref{eq: objfunc1}) and the last player does not control, i.e., $\alpha_\cdot^N \equiv 0$.
We make the same ansatz as in (\ref{ansatz-general}) with $\psi_t^{N,i}=0$ for all $i$. Then the  martingale terms give the same processes 
$Z_t^{i,i,k}$ as in (\ref{martingale-general}) for $i \le N$, $k \le N$, $0 \le t \le T$.\\
From the drift terms, we get the system of Riccati equations :\\
for $i\leq N-1$,
\[
\begin{array}{rll}
     \Dot{\phi}_t^{N,i,i}&=\phi_t^{N,i,i}\cdot \phi_t^{N,i,i}-\epsilon, &\phi_T^{N,i,i}=c, \\
     \Dot{\phi}_t^{N,i,i+1}&=\phi_t^{N,i,i}\cdot \phi_t^{N,i,i+1}+\phi_t^{N,i,i+1}\cdot \phi_t^{N,i+1,i+1}+\epsilon, &\phi_T^{N,i,i+1}=-c,\\
     &\vdots &\\
     \Dot{\phi}_t^{N,i,l}
     &=\phi_t^{N,i,i}\cdot \phi_t^{N,i,l}+\phi_t^{N,i,i+1}\cdot\phi_t^{N,i+1,l}+\cdots+\phi_t^{N,i,l-1}\cdot\phi_t^{N,l-1,l}+\phi_t^{N,i,l}\cdot\phi_t^{N,l,l}, & \phi_T^{N,i,l}=0,\\
     &\vdots&\\
     \Dot{\phi}_t^{N,i,N-1}
     &=\phi_t^{N,i,i}\phi_t^{N,i,N-1}+\phi_t^{N,i,i+1}\phi_t^{N,i+1,N-1}+\cdots+\phi_t^{N,i,N-1}\phi_t^{N,N-1,N-1}, & \phi_T^{N,i,N-1}=0,\\
     \Dot{\phi}_t^{N,i,N}
     &=\sum\limits_{j=i}^{N-1} \phi_t^{N,i,j}\phi_t^{N,j,N}\\
     &=\phi_t^{N,i,i}\phi_t^{N,i,N}+\phi_t^{N,i,i+1}\phi_t^{N,i+1,N}+\cdots+\phi_t^{N,i,N-1}\phi_t^{N,N-1,N}, & \phi_T^{N,i,N}=0;
\end{array}
\]
for $i=N$,
\[
\Dot{\phi}_t^{N,N,N}=-\epsilon, \quad \phi_T^{N,N,N}=c, 
\]
From above, it is demonstrated again that the boundary condition does not affect the solutions $\phi_\cdot^{N,i,j}$ ($j<N$), however, the functions $\phi_\cdot^{N,i,N}$ for all $i$ are different from those in Section \ref{BC-attractedto0}, which are dependent on the boundary condition.

\subsection{Closed-loop Nash Equilibrium}
In search for closed-loop Nash equilibria, the controls are of the form $\alpha^k(t,x)$. When computing $\partial_{x^j}H^i$ in the derivation of the BSDE for $Y^{i,j}$, one needs to pay attention in taking derivatives with respect to $x^j$ in  $\hat\alpha^k$ for $k\neq i$, using  $\hat\alpha^k=-y^{k,k}$ and the ansatz \eqref{eq15}. This is a tedious but straightforward computation which leads to the fact that the obtained closed-loop equilibrium  coincides with the open-loop equilibrium identified  before. We omit the details here as well as repeating this remark in the following sections. The only place where closed-loop and open-loop equilibria will be different is in Section \ref{section 5} when we will look at a mixture of directed chain and mean field interactions for  finite player games, as it is already the case for pure mean field interaction studied in \cite{CarmonaFouqueSunSystemicRisk}. However, they will coincide again for the infinite-player games in Section \ref{section 6}.

\section{Infinite-Player Game Model} \label{section 3}
Motivated by the limit of the finite-player game discussed in Section \ref{section 2}, we define the game with infinite players on a directed chain structure as shown in Figure \ref{fig:chain}. In Remark \ref{finiteYs} in Section \ref{infinite-game-nash}, we will see that the Hamiltonian only depends on finite players, which will make it well-defined.
We assume that the state dynamics of all players are given by the stochastic differential equations of the form: for $i \ge 1$, 
\begin{equation} \label{eq: inf player game states}
{\mathrm d}X_t^i=\alpha^i_t {\mathrm d}t+\sigma {\mathrm d}W_t^i, \quad 0\leq t\leq T , 
\end{equation}
where  $(W_t^i)_{0\leq t\leq T}$, $i\geq 1$ are one-dimensional, independent Brownian motions. Similar to the setup for the finite-player games in Section \ref{section 2}, we assume that the drift coefficients $\alpha^i$ are adapted to the filtration of the Brownian motions and satisfy $\mathbbm{E}[\int_0^T |\alpha_t^i|^2 {\mathrm d}t]<\infty$. We also assume that the diffusion coefficients are constant and identically denoted by $\s>0$. The system starts at time $t = 0$ from $i.i.d.$ square-integrable random variables $X_0^i = \xi_i$ with $ {\mathbbm E}(\xi_i) = 0$, independent of the Brownian motions. 
In this model, player $i$ chooses its own strategy $\alpha^i$ in order to minimize its expected cost function of the form: 
\begin{equation} \label{eq: costInfPlayGame}
J^i(\boldsymbol{\alpha})=\mathbbm E\Big[ \int_0^T f^i(X_s,\alpha_s^{i}) {\mathrm d}s+g^i(X_T) \Big],
\end{equation}
where the running and terminal cost functions $f^{i}(x, \alpha^{i}) $, $g^{i}(x)$ are the same as in  (\ref{eq: fg1}).  

\subsection{Open-Loop Nash Equilibrium}\label{infinite-game-nash}

We search for an open-loop Nash equilibrium of the infinite system (\ref{eq: inf player game states}) among admissible strategies $\{\alpha_t^i, i=1,2,\cdots, 0 \le t \le T\}$.
First, we define the Hamiltonian $H^{i}$ of the form:
\begin{equation} \label{eq: Hamiltonian inf game}
    H^i(x^1,x^2,\cdots,y^{i,1},\cdots, y^{i,n_i},\a^1,\a^2,\cdots)=\sum\limits_ {k=1}^{n_i} \a^k y^{i,k}+\frac{1}{2}(\a^i)^2+\frac{\e}{2}(x^{i+1}-x^i)^2,
\end{equation}
assuming it is defined on real numbers $x^{i}$, $y^{i,k}$, $\alpha^{i}$, $i \ge 1$, $k \ge 1$, where only finitely many $y^{i,k}$ are non-zero for every given $i$. 
Here, $n_i$ is a finite number depending on $i$ with $n_i>i$. This assumption is checked in Remark \ref{finiteYs} below. Thus, the Hamiltonian $H^i$ is well defined for $i \ge 1$.

The adjoint processes $Y_t^i=(Y_t^{i,j};j=1,\cdots,n_i)$ and $Z_t^i=(Z_t^{i,j,k};1\leq j\leq n_i,k\geq 1)$ for $i\ge 1$ are the solutions of the system of backward stochastic differential equations (BSDEs): for $0 \le t \le T$, $i \ge 1$, $1 \le j \le n_{i}$, 
\begin{equation}\label{eq13}
    \left\{
  \begin{array}{ll}
    dY_t^{i,j}&= -\partial_{x^j}H^i(X_t,Y_t^i,\a_t){\mathrm d}t+ \displaystyle \sum_ {k=1}^\infty Z_t^{i,j,k}{\mathrm d}W_t^k\\
    &=-\e (X_t^{i+1}-X_t^i)(\d_{i+1,j}-\d_{i,j}){\mathrm d}t+ \displaystyle \sum_ {k=1}^\infty Z_t^{i,j,k}{\mathrm d}W_t^k,\\
    Y_T^{i,j}&=\partial_{x^j}g_i(X_T)=c(X_T^{i+1}-X_T^i)(\d_{i+1,j}-\d_{i,j}).
  \end{array}
\right.
\end{equation}

\begin{remark}\label{finiteYs}
For every $j\neq i$ or $i+1$, $dY_t^{i,j}=\sum_ {k=1}^\infty Z_t^{i,j,k}dW_t^k$ and $Y_T^{i,j}=0$ implies $Z_t^{i,j,k}=0$ for all $k$. This observation is consistent with \eqref{eq: YijZijk = 0} in the finite player game case. Note also that $Y^{i,i+1}=Y^{i,i}$. There must be finitely many non-zero $Y^{i,j}$'s for every $i$. Hence, the Hamiltonian $H^{i}$ in \eqref{eq: Hamiltonian inf game} can be rewritten as 
\begin{equation*}
    H^i(x^1,x^2,\cdots,y^{i,i},y^{i,i+1},\a^1,\a^2,\cdots)= \a^i y^{i,i}+\a^{i+1} y^{i,i+1}+\frac{1}{2}(\a^i)^2+\frac{\e}{2}(x^{i+1}-x^i)^2.
\end{equation*} 
\end{remark}

\bigskip
By minimizing the Hamiltonian with respect to $\a^i$, we may obtain  
$\Hat{\a}^i=-y^{i,i}$ for all $i$. 
Inspired by the conclusion from the finite-player game (see also Remark \ref{rem: dependence j-i}), we then make the ansatz of the form:
\begin{equation}\label{eq15}
    Y_t^{i,i}=\slimj \phi_t^{j-i}X_t^j, \quad 0 \le t \le T 
\end{equation}
for some deterministic scalar functions $\phi_t^{i}$ satisfying the terminal conditions: $\phi_T^0=c,\phi_T^1=-c,\phi_T^i=0$ for $i\geq 2$.
Substituting the ansatz (\ref{eq15}), the optimal strategy $\hat{\alpha}^{i}$  and the forward equation for $X_{\cdot}^{i}$ in (\ref{eq: inf player game states}) are 
\begin{equation}\label{eq16}
    \Hat{\a}_t^i=-Y_t^{i,i}=-\slimj \phi_t^{j-i}X_t^j, 
    \quad {\mathrm d} X_t^i=-\slimj \phi_t^{j-i}X_t^j {\mathrm d}t+\s {\mathrm d}W_t^i 
\end{equation}
for $i \ge 1$, $0 \le t \le T $. Differentiating the ansatz (\ref{eq15}), we obtain 
\begin{equation}\label{eq17}
  \begin{array}{ll}
    {\mathrm d}Y_t^{i,i}&=\slimj [X_t^j\Dot{\phi}_t^{j-i} {\mathrm d}t+\phi_t^{j-i}{\mathrm d}X_t^j]\\
    &= \sum\limits_{\ell=0}^{\infty} \Dot{\phi}_t^{\ell}X_t^{i+\ell}{\mathrm d}t-\sum\limits_{\ell=0}^{\infty} \Big(\sum\limits_{j=0}^\ell \phi_t^j\phi_t^{\ell-j}\Big) X_t^{i+\ell}{\mathrm d}t+\s \sum\limits_{\ell=i}^\infty \phi_t^{\ell-i}{\mathrm d}W_t^\ell.
  \end{array}
\end{equation}
\\
Now we compare the two It\^o's decompositions (\ref{eq17}) and (\ref{eq13}) of $Y_t^{i,i}$. The martingale terms give the processes 
$Z_t^{i,j,k}$:
\begin{equation*}
    Z_t^{i,i,k}=0 \quad \text{ for } \quad k<i \quad \text{ and } \quad \ Z_t^{i,i,k}=\s\phi_t^{k-i} \quad \text{ for } \quad k\geq i.
\end{equation*}

And from the drift terms, we get the system of Riccati equations: for $0 \le t \le T$
\begin{equation} \label{eq18}
\begin{array}{rll}
     \text{for}\ j=0:&\Dot{\phi}_t^0=\phi_t^0\cdot \phi_t^0-\e , &\phi_T^0=c, \\
     \text{for}\ j=1:& \Dot{\phi}_t^1=2\phi_t^0\cdot \phi_t^1+\e , &\phi_T^1=-c,\\
     \text{for}\ j\geq 2:&\Dot{\phi}_t^j=\phi_t^0\cdot \phi_t^j+\phi_t^1\cdot\phi_t^{j-1}+\cdots+\phi_t^{j-1}\cdot\phi_t^1+\phi_t^j \cdot\phi_t^0 , & \phi_T^j=0.
\end{array}
\end{equation}

The solutions to this Riccati system coincide with the limit of the solutions to the ODE system (\ref{finite-chain-riccati}) of the N-player directed chain game in Section \ref{section 2}, i.e., $\phi^i_{\cdot}=\lim\limits_{N\to\infty} \phi^{N,i,i+j}_{\cdot}$ in the supremum norm. The Riccati system (\ref{eq18}) is solvable.  

\begin{Lemma}\label{inf_sumo}
With the positive constants $c > 0 $, $\varepsilon > 0 $, the solution to the Riccati system (\ref{eq18})  satisfies 
\begin{equation}
\sum\limits_{j=0}^{\infty} \phi_t^{j}=0,\quad\quad
\phi_{t}^{0}=\dfrac{(-\e-c\sqrt{\e})e^{2\sqrt{\e}(T-t)}+\e-c\sqrt{\e}}{(-\sqrt{\e}-c)e^{2\sqrt{\e}(T-t)}-\sqrt{\e}+c} > 0 ,
\end{equation}
for $0\le t \le T$. Moreover, the functions $\phi^k$'s are obtained by a series expansion  of the generating function $S_t(z) = \sum_{k=0}^{\infty} z^{k} \phi^{k}$, $ z \le 1$ of the sequence $\{\phi^{\ell}\}$ given by $S_{t} (1) \equiv 0 $, and if $z < 1 $, 
\begin{equation}
S_{t}(z) = \dfrac{\big(-\e(1-z)-c\sqrt{\e(1-z)}(1-z)\big)\, e^{2\sqrt{\e(1-z)}(T-t)}+\e(1-z)-c\sqrt{\e(1-z)}(1-z)}{\big(-\sqrt{\e(1-z)}-c(1-z)\big)\, e^{2\sqrt{\e(1-z)}(T-t)}-\sqrt{\e(1-z)}+c(1-z)}
\end{equation}
for every $0 \le t \le T$. 
\end{Lemma}
\begin{proof} Given in Appendix \ref{appendix1}. \end{proof}
\begin{remark}
It follows from Lemma \ref{inf_sumo} that the forward dynamics (\ref{eq16}) can be formally written as:
\begin{equation} \label{eq: OLNashInf1}
\begin{split}
      {\mathrm d}X_t^i&=- \sum_{j=0}^\infty\phi_t^{j}X_t^{i+j} {\mathrm d}t+\s {\mathrm d}W_t^i 
     =\phi_t^{0}\cdot \Big( \sum_{j=1}^\infty \frac{-\phi_t^{j}}{\phi_t^{0}}X_t^{i+j} -X_t^{i}\Big) {\mathrm d}t +\s {\mathrm d}W_t^i 
 \end{split} 
\end{equation}
for $i \ge 1$, $0 \le t \le T$. This is a mean-reverting type process, since $\phi_t^0>0$.
We also see that this system is invariant under the shift of indices of individuals. In particular, the law of $X^i$ is the same as the law of $X^1$ for every $i$ and also $X^i$ is independent of $(W^1,\cdots,W^{i-1})$.
\end{remark}

We end with a summary of this section on the infinite player game.  

\begin{prop} An open-loop Nash equilibrium for the infinite-player stochastic game with cost functionals (\ref{eq: costInfPlayGame}) with  (\ref{eq: fg1}) is determined by (\ref{eq: OLNashInf1}), where $\{ \phi^{j}, j \ge 0 \} $ are the unique solution to the infinite system  (\ref{eq18}) of Riccati equations.  
\end{prop}

\section{Catalan Markov Chain}\label{section 4}

In order to simplify our analysis, we look at the stationary solution $\{ \phi^{j} , j \ge 0\}$ of the Riccati system (\ref{eq18}) in Section \ref{section 3}, as $T \to \infty$. 
Without loss of generality, we assume $\epsilon=1$. By taking $T\to \infty$,  we obtain the stationary long-time behavior satisfying $\Dot{\phi}_\cdot^{j}=0$ for all $j$.
Then,  (\ref{eq18}) gives the recurrence relation for the stationary solution $\{ \phi^{j}, j \ge 0 \} $: 
\begin{equation}
\phi^0=1 \quad \text{ and } \quad \sum\limits_{j=0}^n\phi^j\phi^{n-j}=0; \quad n \ge 0 . 
\end{equation} 
This is closely related to the recurrence relation of {\it Catalan} numbers. By using a moment generating function method as in Appendix  \ref{appendix1}, we  get the stationary solutions 
\begin{equation} \label{eq: phi01j}
 \phi^0=1, \quad \phi^1=-\frac{1}{2}, \quad  \phi^j=-\dfrac{(2j-3)!}{(j-2)!\, j!\, 2^{2j-2}} \, \quad \text{ for } \, \, j\geq 2. 
\end{equation}


We consider the continuous-time Markov chain $M(\cdot)$ with state space $\, \mathbb N^\star \,$ and generator matrix 
\begin{equation} \label{eq: Q:sec4}
\mathbf{Q} \, =\,  (q_{i,j}) \,=\left( \begin{array}{ccccc} 
-1 & p_1 &  p_2 & p_3 & \cdots \\
0 & -1 & p_1 &  p_2 & \ddots \\
0 &0 & -1 & p_1 &\ddots\\
& \ddots & \ddots &\ddots  & \ddots\\
\end{array} 
\right) , 
\end{equation} 
where $(i,j)$ element $q_{i,j}$ of $\mathbf Q$ is given by $q_{i,j} = p_{j-i} \cdot {\bf 1}_{\{j \ge i \}}$ with $p_{k} = - \phi^{k} $,  $k \ge 0$, $i, j \ge 1$. Note that the transition probabilities of the continuous-time Markov chain $M(\cdot)$, called a Catalan Markov chain, are  $p_{i,j}(t)=\, \mathbb P (M(t)=j|M(0)=i)=(e^{t\mathbf{Q}})_{i,j},\, i,j \ge 1, \, t\geq 0$. Then with replacement of $\phi^{j}_{t}$, $t\ge 0 $ by the stationary solution $\phi^{j}$ in (\ref{eq16}), the infinite particle system $ (X_{\cdot}^{i}$, $i \ge 1 )$ can be represented formally as a linear stochastic evolution equation:
\begin{equation}  \label{eq22}
{\mathrm d}\mathbf{X}_t=\mathbf{Q\,X}_t {\mathrm d}t+ {\mathrm d}\mathbf{W}_t; \quad t \ge 0 , 
\end{equation}
where $\mathbf{X_.}=(X_{.}^{i}, i \ge 1 )$ with $\mathbf{X}_0=\mathbf{x}_0$ and $\mathbf{W_.}=(W_{.}^{k} , k \ge 1 )$. Its solution is:
\begin{equation}
\mathbf{X}_t=e^{t\mathbf{Q}}\mathbf{x_0}+ \int_0^t e^{(t-s)\mathbf{Q}}{\mathrm d}\mathbf{W}_s ;\quad t\geq 0.
\end{equation}

Without loss of generality, let us assume $\mathbf{X}_0=\mathbf{0}$. Then, 
\begin{equation}
\begin{split} 
{X}_t^{i}&=\mathop{\mathlarger{\int^t_0}}\,  \sum_{j=i}^\infty p_{i,j}(t-s) {\mathrm d} W_{s}^{j} 
 =\mathop{\mathlarger{\int^t_0}}\, \sum_{j=i}^\infty\, \mathbb P (M(t-s)=j|M(0)=i){\mathrm d} W_{s}^{j}\\
 &=\mathbbm{E}^M \Big[ \mathop{\mathlarger{\int^t_0}}\,\sum_{j=i}^\infty\, \mathbf{1}_ {(M(t-s)=j)} {\mathrm d} W_{s}^{j} | M(0)=0 \Big];\quad t\geq 0,
\end{split} 
\end{equation}
where the expectation is taken with respect to the probability induced by the Catalan Markov chain $M(\cdot)$, independent of the Brownian motions $(W_{\cdot}^{j},\, j\in \mathbb{N}_0)$. 
This is a Feynman--Kac representation formula for the infinite particle system $\mathbf X_{\cdot}$ in (\ref{eq22}) associated with the continuous-time Markov chain $M(\cdot)$ with the generator $\mathbf{Q}$. Interestingly, we may compute quite explicitly the corresponding transition probability $(p_{i,j}(\cdot))$ 
for the generator $\mathbf Q$ in (\ref{eq: Q:sec4}). \\

\begin{prop}\label{sol_markovc}
The Gaussian process $\, X_{t}^{i} \,$, $\, i \ge 1 \,$, $\, t \ge 0 \,$ in \eqref{eq22}, corresponding to the Catalan Markov chain, is 
\begin{equation} \label{solution_mc}
\begin{split}
{X}_{t}^{i}  \, &=\, \sum\limits_{j=i}^{\infty}\int^{t}_{0}   (\exp ( Q (t-s)))_{i,j} {\mathrm d} {W}_{s}^{j} \, =\, \sum\limits_{j=i}^{\infty} \int^{t}_{0}  \frac{\,(t-s)^{2(j-i)}  \,}{\, (j-i)!\,} \cdot F^{(j-i)}(-(t-s)^{2})  {\mathrm d} W_{s}^{j}  \\
&\, =\,  \sum\limits_{j=i}^{\infty} \int^{t}_{0}  \frac{\,(t-s)^{2(j-i)}  \,}{\, (j-i)!\,} \cdot \rho_{j-i}(- (t-s)^{2}) \, e^{-(t-s)} \cdot  {\mathrm d} W_{s}^{j},
\end{split}
\end{equation}
where $\, {W}_{\cdot}^{j}  \,$, $ j \in \mathbb N \,$ are independent standard Brownian motions and $\rho_{i}(\cdot)$ is defined by  
\begin{equation} \label{eq: rhokxProp1}
\rho_{i}(x) =\frac{1}{2^i} \sum\limits_{j=i}^{2i-1}\, \frac{(i-1)!}{(2j-2i)!!(2i-j-1)!} \cdot (-x)^{\,-\frac{j}{2}}, 
\end{equation} 
for $i\geq 1$, and $\rho_{0}(x) \, =\,  1$ for $x \le 0 $.  
\end{prop}
\begin{proof}
Given in \textbf{Appendix} \ref{CatalanMC-1}.
\end{proof}

\subsection{Asymptotic Behavior of the Variances as $t \to \infty$}
It follows from (\ref{solution_mc}) that for $\, t \ge 0 \,$, the variance of the Gaussian process $X_{\cdot}^{i}$, $i \ge 1$, in \eqref{eq22} is given by  
\begin{equation} \label{eq: Var4.1}
\begin{split}
\text{Var} ( X_{t}^{i}) = \text{Var} ( X_{t}^{1}) \, &=\, \text{Var} \Big(  \sum_{j=1}^{\infty} \int^{t}_{0}  \frac{\,(t-s)^{2(j-1)} \,}{\,(j-1)!\,} F^{(j-1)}(-(t-s)^{2}) {\mathrm d} W_{s}^{j} \Big)  
\\
\, &=\,  \sum_{j=0}^{\infty} \int^{t}_{0} \frac{\,(t-s)^{4j}\,}{\,(j!)^{2} \,} \lvert  \rho_{j} (- (t-s)^{2}) \rvert^{2} e^{-2(t-s)}{\mathrm d} s.
\end{split}
\end{equation}

\begin{remark}\label{rho_k and var}
To evaluate the variance, we need some estimates of $ \rho_{j}(\cdot)\,$, $\, j \in \mathbb N\,$ in \eqref{eq: rhokxProp1}. It can be shown that
\begin{equation} \label{eq: rho-k-nu2}
\rho_j(-\nu^2)=\frac{\,1\,}{\,2^{j} \nu^{j}\,}\cdot \sqrt{\frac{\,2 \nu \,}{\,\pi \,} } \cdot  e^{\nu} \cdot K_{j-(1/2)} (\nu ) \, ; \quad j \ge 1 \, ,  
\end{equation}
where $K_n(x)$ is the modified Bessel function of the second kind defined by 
\[
K_{n}(x)\, =\,  \int^{\infty}_{0} e^{-x \cosh t} \cosh (nt ) {\mathrm d} t \, ; \quad n > -1 ,  x > 0 .  
\]

Then substituting \eqref{eq: rho-k-nu2} into \eqref{eq: Var4.1}, we obtain 
\begin{equation} \text{\rm Var} ( X_{t}^{1})=\, \sum_{k=1}^{\infty}  \int^{t}_{0} \frac{\,2\,}{\,\pi\,}  \frac{\,\nu^{2k+1}\,}{\,(k!)^{2}\, 4^{k}\,} \big( K _{k-(1/2)}(\nu) \big)^{2} {\mathrm d} \nu + \frac{\,1 - e^{-2t}\,}{\,2\,}; \quad t \ge 0 . 
\end{equation}
Details are given in the Appendix \ref{appendix3}.
\end{remark}

\begin{prop}\label{finite_limit_var} As $t\to \infty $, the asymptotic variance $\lim_{t\to \infty }\text{\rm Var} ( X_{t}^{1})$ is $1/ \sqrt{ 2 } $.  
\end{prop}
\begin{proof}
Given in Appendix \ref{appendixvariance}.
\end{proof}

\subsection{Asymptotic Independence}

With $X_0=0$, it follows from Proposition \ref{sol_markovc} and Remark \ref{rho_k and var} that:
\begin{equation}\label{xforcovariance}
\begin{split}
X_{t}^{i}
&=\sum\limits_{j=0}^\infty \int_0^t \frac{1}{\sqrt{\pi} j!}  \frac{(t-s)^{i+1/2}}{2^{j-1/2}} K_{j-1/2}(t-s) {\mathrm d} W_{j+i}(s); \quad t \ge 0 . 
\end{split}
\end{equation}
Then the auto-covariance and cross-covariance are given respectively by: 
\begin{equation}
\begin{split}
&\mathbbm{E}[X_s^{1}X_t^{1}]=\sum\limits_{j=0}^\infty \int_0^s \frac{1}{\pi (j!)^2 2^{2j-1}}  ((t-s+u)u)^{j+1/2}  K_{j-1/2}(t-s+u) K_{j-1/2}(u) {\mathrm d} u ,\quad 0\leq s\leq t, \label{autocov}\\
&\mathbbm{E}[X_t^{1}X_t^{j+1}]= \sum\limits_{\ell=0}^\infty  \frac{1}{\pi (j+\ell)!\ell!} \frac{1}{2^{j+2\ell-1}}  \int_0^t s^{j+2\ell+1} K_{j+\ell-1/2}(s) K_{\ell-1/2}(s) {\mathrm d} s,\quad t\geq 0.
\end{split}
\end{equation}


Their asymptotic behaviors are summarized in the following. The proofs are given in Appendix \ref{appendixcovariance}.

\begin{prop}[Asymptotic behavior of the auto-covariance] \label{aurocov_sim} 
According to  \eqref{autocov}, the auto-covariance $\mathbbm{E}[X_s^{1}X_t^{1}]$ is positive since $K_n(x)>0$. Fixing $s>0$, when $t-s\to \infty$, it converges to $0$, i.e., the process is ergodic. 
\end{prop}

\begin{prop}[Asymptotic behavior of the cross-covariance] \label{crosscov_sim} 
Similarly, for every $k \ge 0 $ and  for any $t>0$ the cross-covariance $\mathbbm{E}[X_t^{0}X_t^{k}]$ is positive, and 
\[
0< \lim\limits_{t\to\infty} \mathbbm{E}[X_t^{0}X_t^{k}]=\sum\limits_{j=0}^\infty  \frac{1}{\pi (k+j)!j!} \frac{1}{2^{k+2j-1}}  \int_0^\infty s^{k+2j+1} K_{k+j-1/2}(s) K_{j-1/2}(s) {\mathrm d} s \le \frac{1}{\sqrt{2}}.
\]
The asymptotic cross-covariance is positive and bounded above, which means the states are asymptotically dependent in the directed chain game.
\end{prop}

\section{Mixture of Directed Chain and Mean Field Interaction on a Finite-player System }\label{section 5}
In the spirit of the paper, we shall look at the game on a mixed system, including the directed chain interaction and the mean field interaction for finite players. This section repeats the same steps as before to analyze the mixed system game for $N$ players. 
We assume the state dynamics of all the payers are of the form: 
\begin{equation*}
    {\mathrm d} X_t^i=\a_t^i {\mathrm d}t+\s {\mathrm d}W_t^i,
\end{equation*}
for $1 \le i \le N$ as in the previous sections.
In this model, player $i$ chooses its own strategy $\alpha^i$ in order to minimize its objective function of the mixed form:
\begin{equation}
\begin{split}
   i\leq N-1:\quad\quad J^i(\a^1,\cdots,\a^N)=\mathbbm{E} \bigg\{\displaystyle \int_0^T \big( \frac{1}{2}(\a^i_t)^2&+u\cdot \frac{\e}{2}(X_t^{i+1}-X_t^{i})^2+(1-u)\cdot \frac{\e}{2}(\bar{X}_t-X_t^i)^2\big) {\mathrm d}t\\
   &+u\cdot \frac{c}{2}(X_T^{i+1}-X_T^{i})^2+(1-u)\cdot\frac{c}{2}(\bar{X}_T-X_T^i)^2\bigg\},
\end{split}
\end{equation}
for some positive constants $\e,c$ and a weight $u\in [0,1]$. Each player optimizes the cost determined by the mixture of two criteria: distance from the neighbor in the directed chain with weight $u$ and distance from the empirical mean with weight $1 - u$. The notation $\bar{X}_t$ is defined as the empirical mean, i.e., $\bar{X}_t= (X_{t}^{1} + \cdots + X_{t}^{N}) / N $ 
for $t \ge  0 $. The running cost function is defined by 
\begin{equation}
f^i(x,\a^i)=\frac{1}{2}(\a^i)^2+u \cdot \frac{\e}{2}(x^{i+1}-x^{i})^2+(1-u)\cdot \frac{\e}{2}(\bar{x}-x^i)^2 
\end{equation}
and the terminal cost function is defined by \begin{equation} g^i(x)=u \cdot \frac{c}{2}(x^{i+1}-x^{i})^2+(1-u)\cdot \frac{c}{2}(\bar{x}-x^i)^2  , 
\end{equation}
where $\bar{x}$ is defined by $\bar{x} = (x_{1}+ \cdots + x_{N})/N$. 
The system is again completed by describing the behavior of player $N$. For simplicity, we consider the boundary condition of the system where $X^N$ is attracted to $0$. Then we can compare the result with that of Section \ref{BC-attractedto0}. The cost functional for player $N$ is given by:
\begin{equation}
\begin{split}
J^N(\alpha^N)=\mathbbm{E} \bigg\{\displaystyle  \int_0^T \big(\frac{1}{2}(\alpha^N_t)^2&+u\cdot \frac{\epsilon}{2}(X_t^N)^2 +(1-u)\cdot \frac{\e}{2}(\bar{X}_t-X_t^N)^2\big) {\mathrm d}t\\
&+u\cdot \frac{c}{2}(X_T^N)^2+(1-u)\cdot \frac{c}{2}(\bar{X}_T-X_T^N)^2\bigg\}.
\end{split}
\end{equation}
The running cost function is defined by 
\begin{equation} 
f^N(x, \alpha^N)=\frac{1}{2}(\alpha^N)^2+u\cdot \frac{\epsilon}{2}(x^N)^2+(1-u)\cdot \frac{\e}{2}(\bar{x}-x^N)^2 
\end{equation} 
and the terminal cost function is defined by 
\begin{equation}
g^N(x)=u\cdot \frac{c}{2}(x^N)^2+(1-u)\cdot \frac{c}{2}(\bar{x}-x^N)^2. 
\end{equation} 
 If $u=1$, the system becomes the directed chain system discussed before. If $u=0$, it becomes a mean-field system where each player is attracted towards the mean of the system.

\subsection{Open-Loop Nash Equilibrium}
We search for an open-loop Nash equilibrium of the system among strategies $\{\a_t^i,i=1,\cdots,N\}$.
The Hamiltonian $H^{i}$ for player $i$ is given by:
\begin{equation*}
    H^i(x^1,\cdots,x^N,y^{i,1},\cdots,y^{i,N},\a^1,\cdots,\a^N)=\sum\limits_{k=1}^N \a^k y^{i,k}+\frac{1}{2}(\a^i)^2+u\frac{\e}{2}(x^{i+1}-x^i)^2 +(1-u)\frac{\e}{2}(\bar{x}-x^{i})^2,
\end{equation*}
and the Hamiltonian $H^{N}$ for player $N$ is given by:
\begin{equation*}
    H^N(x^1,\cdots,x^N,y^{i,1},\cdots,y^{i,N},\a^1,\cdots,\a^N)=\sum\limits_{k=1}^N \a^k y^{i,k}+\frac{1}{2}(\a^i)^2+u\frac{\e}{2}(x^{N})^2 +(1-u)\frac{\e}{2}(\bar{x}-x^{i})^2.
\end{equation*}

The value of $\alpha^i$ minimizing the Hamiltonian with respect to $\alpha^i$ is given by:
\begin{equation*}
\partial_{\alpha^i} H^i=y^{i,i}+\alpha^i=0\quad \text{leading to the choice:}\quad \hat{\alpha}^i=-y^{i,i}.
\end{equation*}
The adjoint processes $Y_t^i=(Y_t^{i,j};j=1,\cdots,N)$ and $Z_t^i=(Z_t^{i,j,k};j=1,\cdots,N,k=1,\cdots,N)$ for $i=1,\cdots,N$ are defined as the solutions of the backward stochastic differential equations (BSDEs):
\begin{equation}
\begin{array}{ll}
 i<N: & \left\{
  \begin{array}{ll}
    {\mathrm d}Y_t^{i,j}&=-\partial_{x^j}H^i(X_t,Y_t^i,\a_t){\mathrm d}t+ \displaystyle \sum_{k=0}^N Z_t^{i,j,k}{\mathrm d}W_t^k\\
    &= \displaystyle - \Big [ u\e (X_t^{i+1}-X_t^i)(\d_{i+1,j}-\d_{i,j})+(1-u)\e(\bar{X}_t-X_t^i) \Big (\frac{1}{N}-\d_{i,j} \Big)\Big ]  {\mathrm d}t+ \sum_{k=0}^N Z_t^{i,j,k} {\mathrm d}W_t^k,\\
    Y_T^{i,j}&=\partial_{x^j}g_i(X_T)=uc(X_T^{i+1}-X_T^i)(\d_{i+1,j}-\d_{i,j})+(1-u)c(\bar{X}_T-X_T^i)(\frac{1}{N}-\d_{i,j}).
  \end{array}
\right.\\
i=N: & \left\{
    \begin{array}{ll}
       {\mathrm d}Y_t^{N,j}&= \displaystyle -\Big [ u\e X_t^{N}\d_{N,j}+(1-u)\e(\bar{X}_t-X_t^N)\Big (\frac{1}{N}-\d_{N,j} \Big)\Big]{\mathrm d}t+\sum\limits_{k=0}^N Z_t^{N,j,k}{\mathrm d}W_t^k,\\
    Y_T^{N,j}&= \displaystyle ucX_T^{N}\d_{N,j}+(1-u)c(\bar{X}_T-X_T^N) \Big (\frac{1}{N}-\d_{N,j}\Big).
    \end{array}
    \right.
\end{array}
\end{equation}
When $j=i$, it becomes:
\begin{equation}\label{ito-decomp-first}
\begin{array}{lll}
 &\left\{
  \begin{array}{ll}
    {\mathrm d}Y_t^{i,i}
    &= \displaystyle \Big [ u\e (X_t^{i+1}-X_t^i)+(1-u)\e(\bar{X}_t-X_t^i)(1-\frac{1}{N})\Big] {\mathrm d}t+\sum_{k=0}^N Z_t^{i,i,k}{\mathrm d}W_t^k,\\
    
    Y_T^{i,i}&= \displaystyle -uc(X_T^{i+1}-X_T^i)-(1-u)c(\bar{X}_T-X_T^i)\Big(1-\frac{1}{N}\Big),\quad i<N \\
  
       {\mathrm d}Y_t^{N,N}&= \displaystyle \Big[ -u\e X_t^{N}+(1-u)\e(\bar{X}_t-X_t^N)\Big(1-\frac{1}{N}\Big)\Big]  {\mathrm d}t+\sum_{k=0}^N Z_t^{N,N,k} {\mathrm d}W_t^k,\\
    Y_T^{N,N}&= \displaystyle ucX_T^{N}-(1-u)c(\bar{X}_T-X_T^N)\Big(1-\frac{1}{N}\Big).
    \end{array}
    \right.
\end{array}
\end{equation}

Considering the BSDE system and the initial condition, we then make the following ansatz with function parameters depending on $N$:
\begin{equation}\label{ansatz-mixed-finite}
    Y_t^{i,i}=u\sum\limits_{j=i}^N \phi_t^{N,i,j}X_t^j-(1-u)(\bar{X}_t-X_t^i)\theta_t^N,
\end{equation}
for some deterministic scalar functions $\phi_t,\theta_t$ satisfying the terminal condition: when $i<N$, $\phi_T^{N,i,i}=c,\phi_T^{N,i,i+1}=-c,\phi_T^{N,i,j}=0$ for $N\geq j\geq i+2$; $\phi_T^{N,N,N}=c$ and $\theta_T^N=c(1-\frac{1}{N})$. For simplicity of notation, we denote $\theta_t=\theta_t^N$.
Using the ansatz \eqref{ansatz-mixed-finite}, the optimal strategy and forward equation become:
\begin{equation*}
    \left\{
  \begin{array}{ll}
    &\Hat{\a}^i=-Y_t^{i,i}=-u\displaystyle \sum_{j=i}^N \phi_t^{N,i,j}X_t^j+(1-u)(\bar{X}_t-X_t^i)\theta_t,\\
    & {\mathrm d} X_t^j=\Big[-u \displaystyle \sum_{k=j}^N \phi_t^{N,j,k}X_t^k+(1-u)(\bar{X}_t-X_t^j)\theta_t \Big]{\mathrm d}t+\s {\mathrm d}W_t^j.
  \end{array}
\right.
\end{equation*}
By taking the average, 
we obtain 
\begin{equation*}
\begin{split}
{\mathrm d}\bar{X}_t&=-u \cdot 
\frac{1}{N} \sum_{j=1}^N\sum_{k=j}^N \phi_t^{N,j,k}X_t^k
{\mathrm d}t+\s \cdot \frac{1}{N}\sum_{j=1}^N {\mathrm d}W_t^j 
=-u \cdot \frac{1}{N} \sum_{k=1}^N(\sum_{j=1}^k \phi_t^{N,j,k})X_t^k {\mathrm d}t+\s \cdot \frac{1}{N}\sum_{k=1}^N {\mathrm d}W_t^k 
\end{split} 
\end{equation*}
and then 
\begin{equation}\label{mean-diff-eqn}
\begin{split}
{\mathrm d}(\bar{X}_t-X_t^i)
&=-u \cdot \frac{1}{N} \sum\limits_{k=1}^{i-1}(\sum\limits_{j=1}^k \phi_t^{N,j,k})X_t^k {\mathrm d}t+u\sum\limits_{k=i+1}^N \Big(\phi_t^{N,i,k}-\frac{1}{N}\sum\limits_{j=1}^k \phi_t^{N,j,k} \Big)X_t^k {\mathrm d}t\\
& \quad {}+\Big(u\phi_t^{N,i,i}-u\frac{1}{N}\sum\limits_{j=1}^i \phi_t^{N,j,i}+(1-u)\theta_t\Big)X_t^i {\mathrm d}t-(1-u)\bar{X}_t\theta_t {\mathrm d}t\\
& \quad {}+\s \Big(\frac{1}{N}\sum\limits_{k=1}^N {\mathrm d}W_t^k-ud	W_t^i \Big).
\end{split} 
\end{equation}
Differentiating the ansatz (\ref{ansatz-mixed-finite}) and using (\ref{mean-diff-eqn}), we obtain 
\begin{equation}\label{ito-dicomp-mix}
  \begin{array}{ll}
    {\mathrm d}Y_t^{i,i}&=u\cdot \displaystyle \sum_{j=i}^N [X_t^j\Dot{\phi}_t^{N,i,j} {\mathrm d}t+\phi_t^{N,i,j}dX_t^j]-(1-u)\cdot \left(\Dot{\theta}_t (\bar{X}_t-X_t^i){\mathrm d}t+\theta_t {\mathrm d}(\bar{X}_t-X_t^i)\right)\\
    &\stackrel{\text{def}}{=}u\cdot \Romannum{1}-(1-u)\cdot\Romannum{2}.
    \end{array}
\end{equation}
For the first term, we have 
\begin{equation*}
\begin{split} 
  \Romannum{1}&=\sum\limits_{j=i}^N [X_t^j\Dot{\phi}_t^{N,i,j} {\mathrm d}t+\phi_t^{N,i,j} {\mathrm d}X_t^j]\\
     &=\sum\limits_{k=i}^N \big( \dot{\phi}_t^{N,i,k}-u\sum\limits_{j=i}^k \phi_t^{N,i,j}\phi_t^{N,j,k}-(1-u)\theta_t  \phi_t^{N,i,k}\big) X_t^k {\mathrm d}t
     +(1-u)\theta_t \sum\limits_{k=i}^N \phi_t^{N,i,k} \cdot\bar{X}_tdt+\s \sum\limits_{k=i}^N\phi_t^{N,i,k}dW_t^k.
 \end{split} 
\end{equation*}
Then, for the second term, we have 
\begin{equation}
\begin{split} 
   \Romannum{2}&=\Dot{\theta}_t(\bar{X}_t-X_t^i) {\mathrm d}t+\theta_t {\mathrm d} (\bar{X}_t-X_t^i) \\
&=-u\theta_t\frac{1}{N} \sum\limits_{k=1}^{i-1}(\sum\limits_{j=1}^k \phi_t^{N,j,k})X_t^k dt+u\theta_t \sum\limits_{k=i+1}^N(\phi_t^{N,i,k}-\frac{1}{N}\sum\limits_{j=1}^k \phi_t^{N,j,k})X_t^k dt\\
& \quad {} -[\dot{\theta}_t-u\theta_t(\phi_t^{N,i,i}-\frac{1}{N}\sum\limits_{j=1}^i\phi_t^{N,j,i} )-(1-u)\theta_t^2]X_t^idt\\
&\quad {} +(\dot{\theta}_t-(1-u)\theta_t^2) \bar{X}_tdt+\s(\frac{1}{N}\sum\limits_{k=1}^NdW_t^k-dW_t^i).
\end{split} 
\end{equation}
 
Thus $ {\mathrm d} Y_{t}^{i,i} = u\cdot \Romannum{1}-(1-u)\cdot\Romannum{2} $  in (\ref{ito-dicomp-mix})
can be written as:
\begin{equation} \label{ito-decomp-second}
\begin{split}
&\sum\limits_{k=1}^{i-1}(u(1-u)\theta_t\frac{1}{N} \sum\limits_{j=1}^k \phi_t^{N,j,k})X_t^k {\mathrm d}t\\
&\quad {} +\sum\limits_{k=i+1}^N [u\dot{\phi}_t^{N,i,k}-u^2\sum\limits_{j=i}^k \phi_t^{N,i,j}\phi_t^{N,j,k}-u(1-u)\theta_t  \phi_t^{N,i,k}-u(1-u)\theta_t(\phi_t^{N,i,k}-\frac{1}{N}\sum\limits_{j=1}^k\phi_t^{N,j,k}) ]X_t^k {\mathrm d}t\\
&\quad {} +[u\dot{\phi}_t^{N,i,i}-u^2(\phi_t^{N,i,i})^2 -2u(1-u)\theta_t\phi_t^{N,i,i}+(1-u)\dot{\theta}_t+u(1-u)\theta_t\frac{1}{N}\sum\limits_{j=1}^i\phi_t^{N,j,i} -(1-u)^2\theta_t^2 ]X_t^i {\mathrm d}t\\
&\quad {} +[ u(1-u)\theta_t \sum\limits_{k=i}^N \phi_t^{N,i,k}-(1-u)\dot{\theta}_t+(1-u)^2\theta_t^2]\bar{X}_t {\mathrm d}t       \\
&\quad {} + u\s  \sum\limits_{k=i}^N\phi_t^{N,i,k} {\mathrm d}W_t^k-(1-u)\s\theta_t (\frac{1}{N}\sum\limits_{k=1}^N {\mathrm d}W_t^k-dW_t^i)     
\end{split} 
\end{equation}

Now we compare the two It\^o's decompositions (\ref{ito-decomp-first}) and (\ref{ito-decomp-second}). The martingale terms give the processes 
$Z_t^{i,j,k}$:
\begin{equation*}
\begin{array}{ll}
&Z_t^{i,i,k}=-(1-u)\s \theta_t\frac{1}{N}\quad\text{ for } \quad k<i,\\
&Z_t^{i,i,i}=u\s\phi_t^{N,i,i}+(1-u)\s\theta_t(1-\frac{1}{N}) \quad \text{ and } \quad \ Z_t^{i,i,k}=u\s\phi_t^{N,i,k}\quad\text{ for }\quad k>i.
\end{array}
\end{equation*}

And from the drift terms, we get the following system of ordinary differential equations:\\
when $i<N$,
\begin{equation}\label{mixed-finite-riccati}
\begin{array}{lll}
 \text{for} \ i:&\displaystyle u\Dot{\phi}_t^{N,i,i}-u^2(\phi_t^{N,i,i})^2-2u(1-u)\theta_t\phi_t^{N,i,i}+(1-u)\Dot{\theta}_t \Big (1-\frac{1}{N} \Big)-(1-u)^2\theta_t^2 \Big(1-\frac{1}{N} \Big)\\
     & \quad  \displaystyle {} +u(1-u)\theta_t\frac{1}{N} \big ( \displaystyle \sum_{j=1}^{i}\phi_t^{N,j,i}+\sum_{k=i}^{N}\phi_t^{N,i,k} \big)=-u\e-(1-u)\e \Big(1-\frac{1}{N} \Big)^2 ,\quad \phi_T^{N,i,i}=c,\\
\end{array} 
\end{equation}
\begin{equation*}
\begin{array}{lll}
\hspace{-2.5cm}     \text{for} \ i+1:&\ u\Dot{\phi}_t^{N,i,i+1}-u^2(\phi_t^{N,i,i}\phi_t^{N,i,i+1}+\phi_t^{N,i,i+1}\phi_t^{N,i+1,i+1}) \\
     & \displaystyle  \quad {} -2u(1-u)\theta_t\phi_t^{N,i,i+1}-(1-u)\dot{\theta}_t\frac{1}{N}+(1-u)^2\theta_t^2\frac{1}{N}\\
     & \displaystyle \quad {} +u(1-u)\theta_t\frac{1}{N} \big (\sum\limits_{j=1}^{i+1}\phi_t^{N,j,i+1}+\sum\limits_{k=i}^{N}\phi_t^{N,i,k} \big)\\
     & = \displaystyle u\e+ (1-u)\e \Big(1-\frac{1}{N} \Big)\frac{1}{N},\quad \phi_T^{N,i,i+1}=-c,\\
\end{array}
\end{equation*}
\begin{equation*}
\begin{array}{lll}
  \displaystyle    \text{for}\ \ell\geq i+2:&\ u\Dot{\phi}_t^{N,i,\ell}-u^2 \displaystyle \sum\limits_{j=i}^l\phi_t^{N,i,j}\phi_t^{N,j,\ell}-2u(1-u)\theta_t\phi_t^{N,i,\ell}-(1-u)\dot{\theta}_t\frac{1}{N}+(1-u)^2\theta_t^2\frac{1}{N}\\
     &\quad \displaystyle  {} +u(1-u)\theta_t\frac{1}{N} (\sum\limits_{j=1}^{l}\phi_t^{N,j,\ell}+\sum\limits_{k=i}^{N}\phi_t^{N,i,k})=(1-u)\e \Big(1-\frac{1}{N} \Big)\frac{1}{N}, \quad \phi_T^{N,i,\ell}=0,\\
     \text{and  }\ \ &\ \displaystyle  u(1-u)\theta_t \sum\limits_{k=i}^N \phi_t^{N,i,k}-(1-u)\Dot{\theta}_t+(1-u)^2\theta_t^2=(1-u)\e \Big(1-\frac{1}{N}\Big),\quad \theta_T=c \Big(1-\frac{1}{N} \Big);
\end{array}
\end{equation*}
 
When $i=N$,
\begin{equation}\label{mixed-finite-riccati-boundary}
\begin{array}{lll}
 
   \displaystyle   & \displaystyle  \ u\Dot{\phi}_t^{N,N,N}-u^2(\phi_t^{N,N,N})^2-2u(1-u)\theta_t\phi_t^{N,N,N}+(1-u)\Dot{\theta}_t \Big(1-\frac{1}{N} \Big) - (1-u)^2\theta_t^2 \Big(1-\frac{1}{N} \Big)\\
     & \displaystyle  \quad {} +u(1-u)\theta_t\frac{1}{N} (\sum\limits_{j=1}^{N}\phi_t^{N,j,N}+\phi_t^{N,N,N})=-u\e-(1-u)\e \Big(1-\frac{1}{N} \Big)^2 ,\quad \phi_T^{N,i,i}=c,\\
   
     \textit{and  }\ \ &\ \displaystyle  u(1-u)\theta_t  \phi_t^{N,N,N}-(1-u)\Dot{\theta}_t+(1-u)^2\theta_t^2=(1-u)\e \Big(1-\frac{1}{N}\Big),\quad \theta_T=c \Big(1-\frac{1}{N} \Big);
\end{array}
\end{equation}

When $u=1$, the systems (\ref{mixed-finite-riccati}) and (\ref{mixed-finite-riccati-boundary}) are exactly what we obtained for finite-player directed chain game in section \ref{section 2}. We have the similar conclusion that the boundary condition does not affect the functions $\phi_t^{N,i,j}\,(j<N)$ for all $i<N$. We can also compare the system (\ref{mixed-finite-riccati}) with  the system (\ref{eq36}) - (\ref{eq39}) we introduce later. Under suitable assumptions, the system (\ref{mixed-finite-riccati}) may converge, as the number $N$ of players  goes to infinity.

\section{Infinite-Player Game Model with Mean-Field Interaction} \label{section 6}

Motivated by Section \ref{section 5} and following Section \ref{section 3}, we can define a game with infinite players on a mixed system, including the directed chain interaction and the mean field interaction. This section searches for an open-loop Nash equilibrium and repeats the same steps as before to analyse the infinite mixed system game. We have a more general Catalan Markov chain and Table \ref{table:1} below shows the asymptotic behaviors of the variances and covariances as $t\to \infty$ for the process with different types of interactions. Comparing it with Table 1 in Detering, Fouque \& Ichiba \cite{Nils-JP-Ichiba2018DirectedChain}, we have similar conclusions except that our asymptotic variance of purely directed chain does not explode.

The game model is given by:
\begin{equation}\label{eq23}
    {\mathrm d}X_t^i=\a_t^i {\mathrm d}t+\s {\mathrm d}W_t^i;\quad i=1,2,\cdots, \quad 0\leq t\leq T, 
\end{equation}
where $ (W_t^i)_{0\leq t\leq T}$, $i\in\mathbb{N}$ are independent standard Brownian motions. We assume the same drift and diffusion coefficients and the initial conditions as the finite-player game.
By choosing $\a_t^i$, player $i$ tries to minimize:
\begin{equation*}
\begin{split}
     J^i(\a^1,\a^2,\cdots)=\mathbbm{E} \bigg\{ \displaystyle \int_0^T  \big( \frac{1}{2}(\a^i_t)^2&+u\cdot \frac{\e}{2}(X_t^{i+1}-X_t^{i})^2+(1-u)\cdot \frac{\e}{2}(m_t-X_t^{i})^2 \big) {\mathrm d}t\\
      &+u\cdot \frac{c}{2}(X_T^{i+1}-X_T^{i})^2+(1-u)\cdot \frac{c}{2}(m_T-X_T^{i})^2\bigg\},
\end{split}
\end{equation*}
for some positive constants $\e$, $c$ and $u\in [0,1]$.
Here, there is an issue in the choice of 
$m_t$. Intuitively, it should come from the finite-player mixed game described in Section \ref{section 5} as 
the limit of $\bar{X}_{\cdot}$ 
as $N\to\infty$. Combined with the fact that we had $\mathbbm{E}\{X^i_t\}$ independent of $i$, it is natural to set $m_t=\mathbbm{E}\{X^i_t\}$ and check afterwards that this mean value does not depend on $i$ de facto after solving the fixed point step.
Note that the case $u=0$ is very particular, and consists in solving the same mean field game problem for every $i$. The case $u=1$ has already been studied in Section \ref{section 3}, and therefore, in what follows, we concentrate on the case $u\in (0,1)$.

\smallskip
\subsection{Open-Loop Nash Equilibrium}
We search for Nash equilibria of the system among strategies $\{\a_t^i,i=1,2,\cdots\}$.
The Hamiltonian for individual $i$ is given by:
\begin{equation} \label{eq24}
    H^i(t, x^1,x^2,\cdots,y^{i,1},y^{i,2},\cdots,\a^1,\a^2,\cdots)=\slim \a^k y^{i,k}+\frac{1}{2}(\a^i)^2+u\frac{\e}{2}(x^{i+1}-x^i)^2 +(1-u)\frac{\e}{2}(m_t-x^{i})^2.
\end{equation}
The adjoint processes $Y_t^i=(Y_t^{i,j};j\geq 1)$ and $Z_t^i=(Z_t^{i,j,k};j\geq 1, k\geq 1)$ for $i=1,2,\cdots$ are defined as the solutions of the backward stochastic differential equations (BSDEs):
\begin{equation}\label{eq26}
    \left\{
  \begin{array}{ll}
    {\mathrm d}Y_t^{i,j}&=-\big\{u\e (X_t^{i+1}-X_t^i)(\d_{i+1,j}-\d_{i,j})+(1-u)\e(m_t-X_t^i)(-\d_{i,j})\big\} {\mathrm d}t+\slim Z_t^{i,j,k} {\mathrm d}W_t^k,\\
    Y_T^{i,j}&=\partial_{x^j}g_i(X_T)=uc(X_T^{i+1}-X_T^i)(\d_{i+1,j}-\d_{i,j})+(1-u)c(m_T-X_T^i)(-\d_{i,j}).
  \end{array}
\right.
\end{equation}
When $j=i$, it becomes:
\begin{equation} \label{eq27}
 \left\{
    \begin{array}{ll}
       {\mathrm d}Y_t^{i,i}&=\big\{u\e (X_t^{i+1}-X_t^i)+(1-u)\e(m_t-X_t^i)\big\} {\mathrm d}t+\slim Z_t^{i,i,k} {\mathrm d}W_t^k,\\
    Y_T^{i,i}&=-uc(X_T^{i+1}-X_T^i)-(1-u)c(m_T-X_T^i).
    \end{array}
    \right.
\end{equation}
\\According to the Pontryagin stochastic maximum principle, by minimizing the Hamiltonian $H^i$ with respect to $\a^i$, we can get the optimal strategy: $\Hat{\a}^i=-y^{i,i}.$
Then the forward equation becomes:
\begin{equation}\label{eq29}
    {\mathrm d} X_t^i=-Y_t^{i,i} {\mathrm d}t+\s {\mathrm d}W_t^i.
\end{equation}
Similar as Carmona, Fouque, and Sun \cite{CarmonaFouqueSunSystemicRisk}, we define $m_t^X=\mathbbm{E}(X_t)$ and $m_t^Y=\mathbbm{E}(Y_t)$. In equilibrium, we have: $m_t^X=m_t$ for $t\leq T$.\\
Taking expectation in  (\ref{eq27}), we have: $dm_t^Y=0$ and $m_T^Y=0$, and hence,  
$m_t^Y=0$ for $t\leq T$.\\
Taking expectation in (\ref{eq29}) we get: $ {\mathrm d}m_t^X=-m_t^Y {\mathrm d} t=0$ and $m_0^X=\mathbbm{E}(\xi)=0$, and hence, 
 $m_t^X=0$ for $t \le T$.
\\
Now we make the ansatz:
\begin{equation}\label{eq30}
    Y_t^{i,i}=u\slimj \phi_t^{j-i}X_t^j-(1-u)(m_t-X_t^i)\psi_t,
\end{equation}
for some deterministic scalar functions $\phi_t,\psi_t$ satisfying the terminal condition $\phi_T^0=c,\phi_T^1=-c,\phi_T^k=0$ for $k\geq 2$ and $\psi_T=c$.
Using this ansatz, the forward equation (\ref{eq23}) becomes 
\begin{equation}\label{eq31}
    \left\{
  \begin{array}{ll}
    &\Hat{\a}^i=-Y_t^{i,i}=-u\slimj \phi_t^{j-i}X_t^j+(1-u)(m_t-X_t^i)\psi_t,\\
    & {\mathrm d}X_t^i= \big( -u\slimj \phi_t^{j-i}X_t^j+(1-u)(m_t-X_t^i)\psi_t \big) {\mathrm d} t+\s {\mathrm d}W_t^i.
  \end{array}
\right.
\end{equation}
\\
Using (\ref{eq31}) and $ {\mathrm d}m_t= {\mathrm d}m_t^X=0$, we can differentiate the ansatz (\ref{eq30}) to obtain 
\begin{equation}\label{eq32}
  \begin{array}{ll}
    {\mathrm d}Y_t^{i,i}&=u\cdot\slimj [X_t^j\Dot{\phi}_t^{j-i}{\mathrm d}t+\phi_t^{j-i} {\mathrm d} X_t^j]-(1-u)\cdot \left(\Dot{\psi}_t (m_t-X_t^i) {\mathrm d}t+\psi_t {\mathrm d}(m_t-X_t^i)\right)\\
    &\stackrel{\text{def}}{=}u\cdot \Romannum{1}-(1-u)\cdot\Romannum{2} 
    \end{array}
\end{equation}
where  
\begin{equation}\label{eq33}
  \begin{array}{ll}
  \Romannum{1}&=\slimj[X_t^j\Dot{\phi}_t^{j-i} {\mathrm d}t+\phi_t^{j-i} {\mathrm d}X_t^j]\\
    &= \slimo\Big( \Dot{\phi}_t^{k}-u\sum\limits_{j=0}^k \phi_t^j\phi_t^{k-j}\Big) X_t^{i+k}{\mathrm d}t+(1-u)\psi_t\slimo\phi_t^{k}(m_t-X_t^{i+k}){\mathrm d}t +\s \sum\limits_{k=i}^\infty \phi_t^{k-i}{\mathrm d}W_t^k \\
     \end{array}
\end{equation}
and 
    \begin{align}
         \Romannum{2}&=\Dot{\psi}_t(m_t-X_t^i) {\mathrm d}t+\psi_t {\mathrm d}(m_t-X_t^i)  
         =u\psi_t\slimo\phi_t^kX_t^{i+k} {\mathrm d}t+(\Dot{\psi}_t-(1-u)\psi_t^2)(m_t-X_t^i) {\mathrm d}t-\psi_t\s {\mathrm d}W_t^i.
    \end{align}

Now we compare the two It\^o's decompositions \eqref{eq32} 
and (\ref{eq27}).
First, the martingale terms give the processes 
$Z_t^{i,j,k}$:
\begin{equation*}
    Z_t^{i,i,k}=0\text{ for } k<i,\ Z_t^{i,i,i}=u\s\phi_t^{0}+(1-u)\s\psi_t \quad \text{ and } \quad \ Z_t^{i,i,k}=u\s\phi_t^{k-i} \quad \text{ for }  \quad k>i.
\end{equation*}

And from the drift terms we obtain the system of ordinary differential equations: 
\begin{align}
     \text{for} \ k=0:&\ u\Dot{\phi}_t^{0}-u^2(\phi_t^0)^2-2u(1-u)\psi_t\phi_t^0+(1-u)\Dot{\psi}_t-(1-u)^2\psi_t^2=-\e ,&\quad \psi_T=c,\,\phi^0_T=c\label{eq36}\\
     \text{for}\ k=1:&\ u\Dot{\phi}_t^1-2u^2\phi_t^0\phi_t^1-2u(1-u)\psi_t\phi_t^1=u\e,\quad &\phi^1_T=-c \label{eq37}\\
     \text{for}\ k\geq 2:&\ u\Dot{\phi}_t^k-u^2\sum\limits_{j=0}^k\phi_t^j\phi_t^{k-j}-2u(1-u)\psi_t\phi_t^k=0, &\phi^k_T=0\label{eq38}\\
     \text{and  }\ \ &\ u(1-u)\psi_t \slimo \phi_t^k-(1-u)\Dot{\psi}_t+(1-u)^2\psi_t^2=(1-u)\e, &\quad \psi_T=c.\label{eq39}
\end{align}
In Appendix \ref{infinite-sum0} we show the following result which simplifies (\ref{eq39}) considerably.
\begin{prop} \label{infinite-sum=0} The solution $\phi_{\cdot}^{j}$ to the system of Riccati equations satisfies 
$\sum_{j=0}^\infty \phi_t^{j}=0 $ for $ t \ge 0 $. 
\end{prop}
Using Proposition \ref{infinite-sum=0} and $0<u< 1$, we can simplify the equations (\ref{eq36}) to (\ref{eq39}):
\begin{equation}
    \begin{array}{lll}
    &\Dot{\psi}_t=(1-u)\psi_t^2-\e,& \psi_T=c\textit{   (Riccati)},\\
    \text{for}\ k=0:&\Dot{\phi}_t^0=u\phi_t^0\cdot \phi_t^0+2(1-u)\psi_t\phi_t^0-\e, &\phi_T^0=c\textit{   (Riccati)}, \\
     \text{for}\ k=1:& \Dot{\phi}_t^1=2u\phi_t^0\cdot \phi_t^1+2(1-u)\psi_t\phi_t^1+\e, &\phi_T^1=-c,\\
     \text{for} \ k\geq 2:&\Dot{\phi}_t^k=u(\phi_t^0\cdot \phi_t^k+\phi_t^1\cdot\phi_t^{k-1}+\cdots+\phi_t^{k-1}\cdot\phi_t^1+\phi_t^k\cdot\phi_t^0)+2(1-u)\psi_t\phi_t^k, & \phi_T^k=0.
    \end{array}
\end{equation}
Looking at the stationary solution (in the limit ($T\to \infty$), and 
without loss of generality assuming $\e=1$ again, the recurrence relation can be solved by the  method of moment generating function to obtain:
\begin{equation} \label{eq: stationary psi phi k}
    \left\{
    \begin{array}{ll}
         &\psi=\sqrt{\frac{1}{1-u}}, \quad \phi^0=\frac{1-\sqrt{1-u}}{u},  \\
         & \phi^1=-\frac{1}{2}, \quad 
         \phi^k=-\dfrac{(2k-3)!}{(k-2)!k!2^{2k-2}} u^{k-1},\quad for\, k\geq 2.
    \end{array}
    \right.
\end{equation}


\subsection{Catalan Markov Chain for the Mixed Model}
As in  Section \ref{section 4}, we consider a continuous-time Markov chain $M^{(u)}(\cdot)$ in the state space $\, \mathbb N \,$ with generator matrix 
\begin{equation}
\mathbf{Q}^{(u)}\,=\left( \begin{array}{ccccc} 
-1 & q_1 &  q_2 & q_3 & \cdots \\
0 & -1 & q_1 &  q_2 & \ddots \\
0 &0 & -1 & q_1 &\ddots\\
& \ddots & \ddots &\ddots  & \ddots\\
\end{array} 
\right) 
\end{equation} 
where $q_{k} = - u \phi^{k} > 0 $ with $\phi^{k}$ in \eqref{eq: stationary psi phi k} for $k \ge 1$. Note that $ \sum_{k=1}^{\infty} q_{k} = 1 - \sqrt{ 1 - u } $. 

For $0<u<1$, $\mathbf{Q}^{(u)}$ is the generator of the Markov chain 
from $i$ to $i+k$ with rate $q_{k}$ and killed with probability $\sqrt{ 1 - u } $. 
The infinite particle system (\ref{eq31}) can be represented as the infinite-dimensional stochastic evolution equation:
\begin{equation}  \label{eq45}
{\mathrm d}\mathbf{X^{(u)}_t}=\mathbf{Q^{(u)}\,X^{(u)}_t}{\mathrm d}t+ {\mathrm d}\mathbf{W}_t,
\end{equation}
where $\mathbf{X^{(u)}_.}=(X_{.}^{i},i \in \mathbb{N} )$ with $\mathbf{X^{(u)}_0}=\mathbf{x^{(u)}_0}$ and $\mathbf{W_.}=(W_{.}^{i}, i\in \mathbb{N})$. 
The solution is formally written by 
\begin{equation}
\mathbf{X^{(u)}_t}=e^{t\mathbf{Q^{(u)}}} \mathbf{x^{(u)}_0}+ \int_0^t e^{(t-s)\mathbf{Q^{(u)}}} {\mathrm d}\mathbf{W}_s ;\quad t\geq 0.
\end{equation}

Note that the transition probabilities of the continuous-time Markov chain $M^{(u)}(\cdot)$ is : $p_{i,j}(t)=\, \mathbb P (M^{(u)}(t)=j|M^{(u)}(0)=i)=(e^{t\mathbf{Q^{(u)}}})_{i,j}$, $ i,j \ge 1\,$, $t\geq 0$. Without loss of generality, assume $\mathbf{x^{(u)}_0}=\mathbf{0}$. Then, 
\begin{equation}
\begin{array}{ll}
X_{t}^{i} 
&= \displaystyle \mathop{\mathlarger{\int^t_0}}\,  \sum\limits_{j=i}^\infty p_{i,j}(t-s) {\mathrm d} W_{s}^{j} 
 = \displaystyle \mathop{\mathlarger{\int^t_0}}\, \sum\limits_{j=i}^\infty\, \mathbb P (M(t-s)=j|M(0)=i) {\mathrm d} W_{s}^{j}\\
 &= \displaystyle \mathbbm{E}^M \Big[ \mathop{\mathlarger{\int^t_0}}\,\sum\limits_{j=i}^\infty\, \mathbf{1}_ {(M(t-s)=j)} {\mathrm d} W_{s}^{j} | M(0)=i \Big];\quad t\geq 0,
\end{array}
\end{equation}
where the expectation is taken with respect to the probability induced by the Markov chain $M^{(u)}(\cdot)$, independent of the Brownian motions $(W_{\cdot}^{k},\, k\in \mathbb{N})$. Therefore, we have a Feynman--Kac representation formula for the generator $\mathbf{Q^{(u)}}$.
Since $\sum\limits_{i=1}^{k-1} q_i q_{k-i}=u^2\sum\limits_{i=1}^{k-1}\phi_t^{(i)}\phi_t^{(k-i)}=-2u\phi_t^{(k)}=2q_k$ we have $\, (\mathbf{Q^{(u)}})^{2} \, =\,  I - uB\,$ with $\, B\,$ having $\,1\,$'s on the upper second diagonal and $\,0\,$'s elsewhere, i.e., 
\begin{equation}
(\mathbf{Q^{(u)}})^{2} \, =\, I - u B = \left( \begin{array}{ccccc} 
1 & -u &  0 & \cdots & \\
0 & 1 & -u &  \ddots &  \\
& \ddots & \ddots &\ddots  & \\
\end{array} 
\right) \, . 
\end{equation} 

\smallskip 

With a smooth function $F(x) \, :=\, \exp ( - \sqrt{ - x } )\, ,  \, \, x \in \mathbb C $, the matrix exponential of $\, \mathbf{Q^{(u)}} t\,$ can be written formally
\[
\exp (\mathbf{Q^{(u)}} t )=\, F((-I+uB)t^2)=\,\sum\limits_{j=0}^\infty \frac{F^{(j)}(-t^2)}{j!}\, (uBt^2)^j=\,\sum\limits_{j=0}^\infty \frac{u^j t^{2j}F^{(j)}(-t^2)}{j!}\, B^j 
\]
for $t \ge 0$. Then the $\,(i,j)\,$-element of $\, \exp ( \mathbf{Q} t ) \,$ is formally given by 
\[
(\exp (\mathbf{Q^{(u)}} t ))_{i,j} \, =\,  \frac{\,u^{j-i} t^{2(j-i)} \cdot F^{(j-i)}(-t^{2}) \,}{\, (j-i)!\,} \,  , \quad i \le j \, , \, \, \text{ where } \, \, F^{(j)} (x) \, :=\, \frac{\,{\mathrm d}^{j}  F \,}{\,{\mathrm d} x^{j}\,}(x)  \, ; \quad x > 0 \,, \, \, j \in \mathbb N \,  ,   
\]
and $\, (\exp ( \mathbf{Q^{(u)}} t))_{i,j} \, =\,  0 \,$, $\, i > j \,$ for $\, t \ge 0 \,$. 

\smallskip 

As in Section \ref{section 4}, we have the same solution for $\, F^{(k)}(x) \,$: $\,F^{(k)}(x) \, =\,  \rho_{k}(x) e^{ - \sqrt{-x}}\,$, where 
$\rho_{k}(\cdot)$ 
was defined in (\ref{eq: rhokxProp1}), and we can summarize our finding: 

\smallskip 

\begin{prop}\label{catalanMC-sol-general}
 The Gaussian process $\, X_t^{i} \,$, $\, i \in \mathbb N\,$, $\, t \ge 0 \,$, corresponding to the (Catalan) general Markov chain, is 
\begin{equation*}
\begin{split}
{X}^{i}_{t}  \, & =\, \sum\limits_{j=0}^{\infty}\int^{t}_{0}   (\exp ( \mathbf{Q^{(u)}} (t-s)))_{i,j} {\mathrm d} {W}_{s}^{j} \, =\, \sum\limits_{j=i}^{\infty} \int^{t}_{0}  \frac{\, u^ {j-i} (t-s)^{2(j-i)}  \,}{\, (j-i)!\,} \cdot F^{(j-i)}(-(t-s)^{2})  {\mathrm d} W_{s}^{j}   
\end{split}
\end{equation*}
\begin{equation}
\begin{split}
\hspace{-4cm} &\, =\,  \sum\limits_{j=i}^{\infty} \int^{t}_{0}  \frac{\, u^{j-i}(t-s)^{2(j-i)}  \,}{\, (j-i)!\,} \cdot \rho_{j-i}(- (t-s)^{2}) \, e^{-(t-s)} \cdot  {\mathrm d} W_{s}^{j},
\end{split}
\end{equation}
where $\, {W}_{\cdot}^{j}  \,$, $\,  j \in \mathbb N \,$ are independent standard Brownian motions. 
\end{prop}

\subsection{Asymptotic Behavior}

Table \ref{table:1} exhibits the asymptotic behaviors of their variances and covariances as $t\to \infty$. The calculation is given in Appendix \ref{table1}. We find that only when $u=0$ (i.e. pure mean field game), the asymptotic cross-covariance is zero, which means the states are asymptotically independent. Otherwise, they are dependent and their covariance is finite. Note that in the purely nearest neighbor interaction studied in Detering, Fouque, and Ichiba \cite{Nils-JP-Ichiba2018DirectedChain}, i.e., in the case $u=0$, the variance is not stabilized as in our ``Catalan" interaction equilibrium dynamics.

\begin{table}[h!]
\begin{center}
 \begin{tabular}{|c| c| c| c|} 
 \hline
 $u$ & Interaction Type & Asymptotic Variance & Asymptotic Independence between two players \\ [0.5ex] 
 \hline\hline
 $u=0$ & Purely mean-field & Stabilized & Independent \\ 
 \hline
  $u\in (0,1)$& Mixed interaction & Stabilized & Dependent \\
 \hline
 $u=1$ & Purely directed chain & Stabilized & Dependent \\[1ex] 
 \hline
\end{tabular}
\end{center}
\caption{Asymptotic behaviors as $t\to \infty$}
\label{table:1}
\end{table}

\section{Periodic Directed Chain Game}\label{section 7}


We consider a stochastic game with finite players on a periodic ring structure. We assume the dynamics of the states of the individual players are given by $N$ stochastic differential equations of the form:
 \begin{equation}\label{eq1}
    {\mathrm d}X_t^i=\a_t^i {\mathrm d}t+\s {\mathrm d}W_t^i,\quad  i=1,\cdots,N, \quad  0\leq t\leq T , 
\end{equation}
where $(W_t^i)_{0\leq t\leq T},\, i = 1, \cdots, N$ are one-dimensional independent standard Brownian motions. The drift coefficient function, the diffusion coefficient and the initial conditions are assumed to be the same as those in Section \ref{section 2}. In this model, player $i$ chooses its own strategy $\alpha^i$ in order to minimize its objective function of the form:
\begin{equation}
    J^i(\a^1,\cdots,\a^N)=\mathbbm{E} \left\{\int_0^T \left(\frac{1}{2}(\a^i_t)^2+\frac{\e}{2}(X_t^{i+1}-X_t^{i})^2\right) {\mathrm d}t+\frac{c}{2}(X_T^{i+1}-X_T^{i})^2\right\},
\end{equation}
for  constants $\e>0$, and $c\geq 0$, and we define $X_\cdot^{N+1}=X_\cdot^{1}$. 

\subsection{Construction of an Open-Loop Nash Equilibrium}
We search for Nash equilibria of the system among strategies $\{\a_t^i,i=1,\cdots,N\}$. We construct an open-loop Nash equilibrium by the Pontryagin stochastic maximum principle.
The Hamiltonian for player $i$ is given by:
\begin{equation}\label{eq3}
    H^i(x^1,\cdots,x^N,y^{i,1},\cdots,y^{i,N},\a^1,\cdots,\a^N)=\sum\limits_ {k=1}^N \a^k y^{i,k}+\frac{1}{2}(\a^i)^2+\frac{\e}{2}(x^{i+1}-x^i)^2.
\end{equation}
The adjoint processes $Y_t^i=(Y_t^{i,j};j=1,\cdots,N)$ and $Z_t^i=(Z_t^{i,j,k};j,k=1,\cdots,N)$ for $i=1,\cdots,N$ are defined as the solutions of the system of the backward stochastic differential equations (BSDEs):
\begin{equation}\label{eq5}
    \left\{
  \begin{array}{ll}
    dY_t^{i,j}&=-\e (X_t^{i+1}-X_t^i)(\d_{i+1,j}-\d_{i,j}){\mathrm d}t+\sum\limits_ {k=1}^N Z_t^{i,j,k}{\mathrm d}W_t^k,\\
    Y_T^{i,j}&=\partial_{x^j}g_i(X_T)=c(X_T^{i+1}-X_T^i)(\d_{i+1,j}-\d_{i,j}).
  \end{array}
\right.
\end{equation}

Based on the sufficiency part of the Pontryagin stochastic maximum principle, we can get an open-loop Nash equilibrium by minimizing the Hamiltonian $H^i$ with respect to $\a^i$:
\begin{equation} \label{eq6}
    \partial_{\a^i}H^i=y^{i,i}+\a^i=0 \quad \textit{leading to the choice:}\quad  \Hat{\a}^i=-y^{i,i}.
\end{equation}
With this choice for the controls $\a^i$'s, the forward equation (\ref{eq1}) becomes coupled with the backward equation (\ref{eq5}). We make the ansatz:
\begin{equation}\label{eq7}
    Y_t^{i,i}=\sum\limits_{j=0}^{N-1} \phi_t^{N,j}X_t^{i+j},
\end{equation}
for some deterministic scalar functions $\phi_t$ satisfying the terminal conditions: $\phi_T^{N,0}=c,\phi_T^{N,1}=-c,\phi_T^{N,k}=0$ for $k\geq 2$ and $X_t^{i+j}\stackrel{def}{=} X_t^{(i+j) \mod N}$.
Using the ansatz, the optimal strategy (\ref{eq6}) and the forward equation (\ref{eq1}) become:
\begin{equation}\label{eq8}
    \Hat{\a}^i=-Y_t^{i,i}=-\sum\limits_{j=0}^{N-1}  \phi_t^{N,j}X_t^{i+j}, 
   \quad  {\mathrm d} X_t^i=-\sum\limits_{j=0}^{N-1} \phi_t^{N,j}X_t^{i+j} {\mathrm d}t+\s {\mathrm d}W_t^i.
\end{equation}
\\
Using the equations (\ref{eq8}), we can differentiate the ansatz (\ref{eq7}):
\begin{equation}\label{eq9}
 \begin{split} 
    {\mathrm d}Y_t^{i,i}&=\sum_{j=0}^{N-1} [X_t^{i+j}\Dot{\phi}_t^{N,j} {\mathrm d}t+\phi_t^{N,j} {\mathrm d}X_t^{i+j}]\\
    &=\sum\_{j=0}^{N-1} X_t^{i+j}\Dot{\phi}_t^{N,j} {\mathrm d}t-\sum\limits_{j=0}^{N-1} \phi_t^{N,j} \sum_{k=0}^{N-1} \phi_t^{N,k}X_t^{i+j+k} {\mathrm d}t+\sum_{j=0}^{N-1}  \s\phi_t^{N,j} {\mathrm d}W_t^{i+j}\\
  \end{split}
\end{equation}
\\
Now we compare the two It\^o's decompositions (\ref{eq9}) and (\ref{eq5}) of $Y_t^{i,i}$. The martingale terms give the processes 
$Z_t^{i,j,k}$:
\begin{equation*}
    Z_t^{i,i,k}=\s\phi_t^{N,N+k-i} \quad \text{ for } \quad 1\leq k<i \quad \text{ and } \quad \ Z_t^{i,i,k}=\s\phi_t^{N,k-i} \quad \text{ for } \quad i\leq k\leq N.
\end{equation*}

And from the drift terms, we get the system of ordinary differential equations 
\begin{equation} \label{eq10}
\begin{array}{lll}
     \text{for}\ k=0:
     &\Dot{\phi}_t^{N,0}=\phi_t^{N,0}\cdot \phi_t^{N,0}+ \displaystyle \sum\limits_{i=1}^{N-1} \phi_t^{N,i} \phi_t^{N,N-i} -\e , &\phi_T^{N,0}=c,\\
     \text{for}\ k=1:
     &  \displaystyle  \Dot{\phi}_t^{N,1}=\phi_t^{N,0}\cdot \phi_t^{N,1}+ \phi_t^{N,1}\cdot\phi_t^{N,0} +\sum\limits_{i=2}^{N-1} \phi_t^{N,i} \phi_t^{N,N+1-i} +\e, &\phi_T^{N,1}=-c,\\
     \text{for}\ N-1>k\geq 2:
    & \displaystyle  \Dot{\phi}_t^{N,k}=\sum\limits_{j=0}^{k} \phi_t^{N,j} \phi_t^{N,k-j} + \sum\limits_{i=k+1}^{N-1} \phi_t^{N,i} \phi_t^{N,N+k-i} ,&\phi_T^{N,k}=0,\\
    \\
     \text{for}\ k=N-1: 
     & \displaystyle  \Dot{\phi}_t^{N,N-1}=\sum\limits_{j=0}^{N-1} \phi_t^{N,j} \phi_t^{N,N-1-j} ,&\phi_T^{N,N-1}=0.
\end{array}
\end{equation}
It can be written as a matrix Ricatti equation: 
\begin{equation} \label{matriRicc1}
\dot{\Phi}^N(t) = \Phi^N(t) \Phi^N(t) - {\mathbf E} \, , \quad \Phi^N(T) := {\mathbf C} \, , 
\end{equation}
where $\Phi^N(\cdot)$ is the $N\times N$ matrix-valued function given by 
$$
\Phi^N (t) := \left( \begin{array}{ccccc}
\phi^{N,0}_t & \phi^{N,N-1}_t & \cdots & & \phi^{N,1}_t \\
\phi^{N,1}_t & \phi^{N,0}_t & \ddots & & \phi^{N,2}_t \\
\vdots & \ddots & \ddots & \ddots & \vdots \\
\vdots & \ddots & \ddots & \ddots & \phi^{N,N-1}_t\\
\phi^{N,N-1}_t & \cdots & & \phi^{N,1}_t & \phi^{N,0}_t \\
\end{array}
\right) \, ,$$ and ${\mathbf E} $  and $\mathbf C$ are constant matrices defined by 
$$  
\mathbf E := \left ( \begin{array}{ccccc} \epsilon & 0 & \cdots & 0 & -\epsilon \\
-\epsilon & \epsilon & \ddots & \ddots & 0 \\
0 & -\epsilon & \ddots & \ddots  &\vdots \\
\vdots & \ddots & \ddots & \ddots & 0 \\
0& \cdots & 0 & -\epsilon &  \epsilon \\ 
\end{array}\right) \, , \quad
\mathbf C := \left ( \begin{array}{ccccc} c & 0 & \cdots & 0 & -c \\
-c & c & \ddots & \ddots & 0 \\
0 & -c & \ddots & \ddots  &\vdots \\
\vdots & \ddots & \ddots & \ddots & 0 \\
0& \cdots & 0 & -c &  c \\ 
\end{array}\right) .  
$$

\begin{prop} \label{finite_sumo}
The solution $\phi_{\cdot}^{N,k} $, $k = 1, \ldots , N$ to the system of Riccati equations \eqref{eq10} satisfies the relation $\sum_{k=0}^{N-1} \phi_t^{N,k}=0$ for $0 \le t \le T$. 
\end{prop}
\begin{proof}
Given in Appendix \ref{appendixsumo}.
\end{proof}

With finite $N$, these equations are not easy to solve explicitely. If we take $N=\infty$, we expect that the system converges to the Riccati system of the infinite-player game studied in Section \ref{section 3}.

\begin{conjecture}\label{conjecturefinite_phi}
The limit of each element in $\Phi^{N}(\cdot)$ in \eqref{matriRicc1} exists as $N\to \infty$, i.e., $\Phi^N (t)\to \Phi^\infty(t)$ and the limit $\Phi^\infty(t)$ is an infinite dimensional, lower triangular, matrix-valued function of $t \ge 0$ given by $\Phi^{\infty}(t) = (\Phi^{\infty, i,j}(t))_{i,j \in \mathbb N}$ with $\Phi^{\infty,i,j} (\cdot) \equiv 0 $ if $i < j$; $\Phi^{\infty,i,j} (\cdot) \equiv \phi^{i-j}$ if $i \ge j$, where the functions $\phi^{k}$'s are given by the system of ordinary differential equations  \eqref{eq18}.
\end{conjecture}


\begin{remark} \label{sum_phi_to_0}
Proving this conjecture is equivalent to show that $\sum_{k=j+1}^{N-1} \phi_t^{N,k}\phi_t^{N,N+j-k}\to 0$ as $N\to\infty$. For instance, for $j=0$, one needs to show that $\sum_{k=1}^{N-1} \phi_t^{N,k}\phi_t^{N,N-k}\to 0$. As of now, this remains an open problem.
\end{remark}

Our conjecture  is substantiated by numerical evidences presented below. 

\subsection{Numerical Results}
Using the methods given in \cite{VanghanIEEE69}, we can get the numerical solution of the matrix Riccati equation (\ref{matriRicc1}). Taking $\epsilon=2,\, c=1,\, T=10$ (large terminal time), Figure \ref{fig:finitephi} (a)-(b) show the behaviors of the $\phi$ functions defined by  the system of differential equations  (\ref{eq18}) for $N=4$ and $N=100$. They converge to the constant solutions of the infinite game given in Section \ref{section 4}, except in the tail close to maturity as $T$ is large but not infinite. This result confirms our conjecture stated in the previous section. Figure 
\ref{fig:finitephi}  (c) shows the behavior of the function $\sum_{k=1}^{N-1} \phi_t^{N,k}\phi_t^{N,N-k}$ for different values of $N$. As we can see, the sum converges to $0$ when $N$ becomes larger, which supports the statement in Remark \ref{sum_phi_to_0}. 
Though these numerical results give us strong evidence and confidence that the conjecture is true, a mathematical proof is still needed and it is part of our ongoing research.

\begin{figure}[H]
  \centering
  \begin{subfigure}[b]{0.33\linewidth}
    \includegraphics[width=\linewidth]{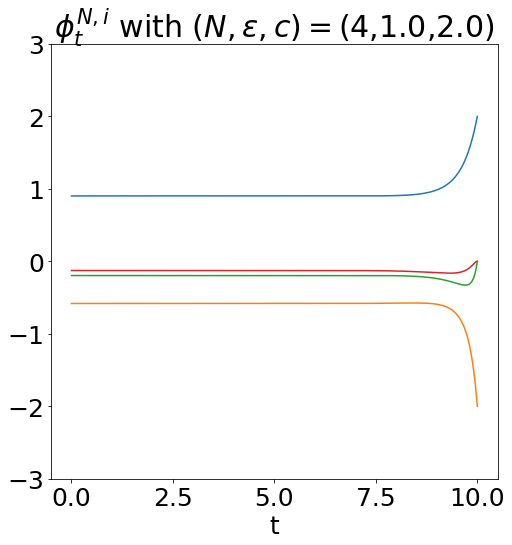}
    \caption{N=4}
  \end{subfigure}
  \begin{subfigure}[b]{0.33\linewidth}
    \includegraphics[width=\linewidth]{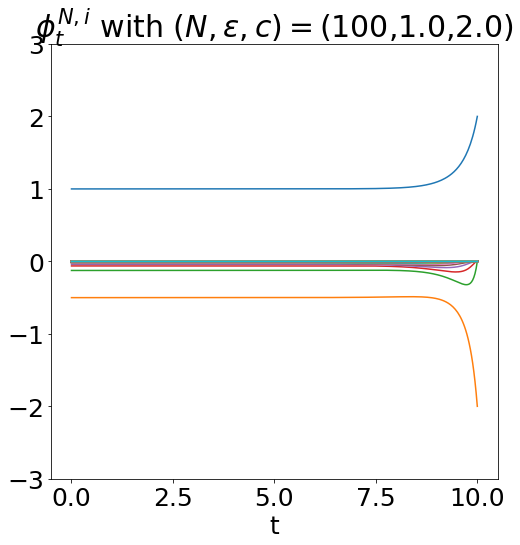}
    \caption{N=100}
  \end{subfigure}
  \begin{subfigure}[b]{0.33\linewidth}
  \includegraphics[scale = 0.16]{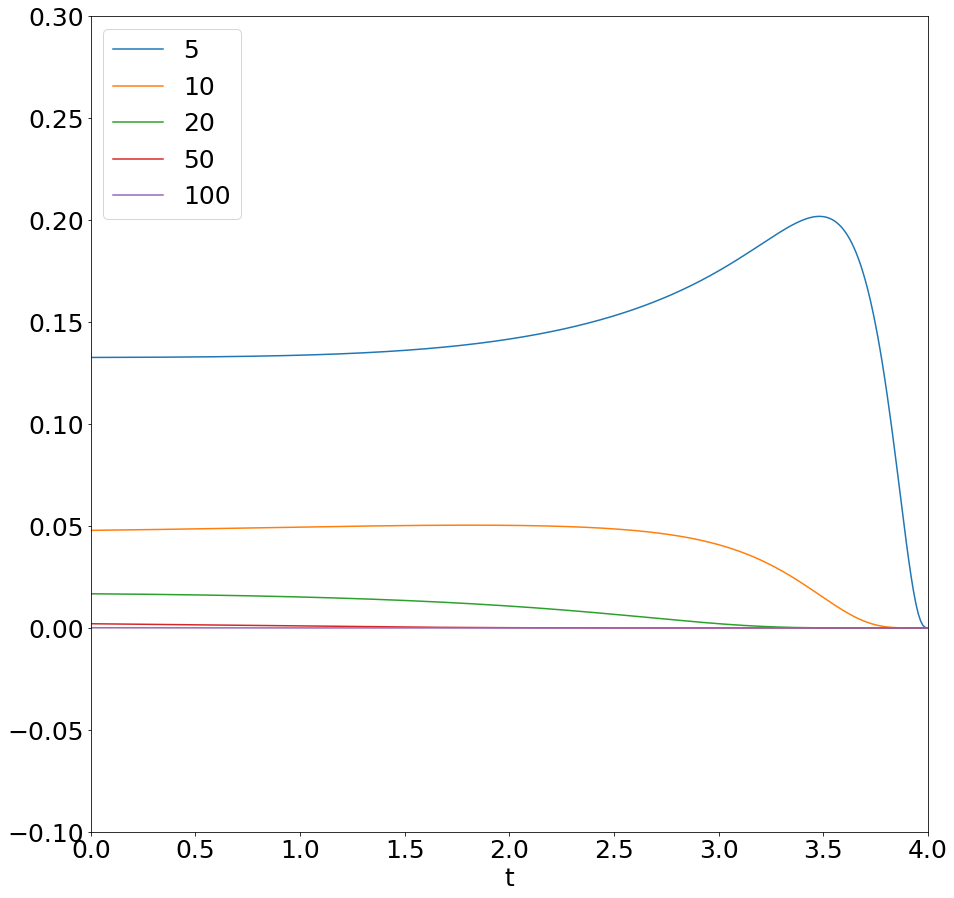}
  \caption{$\sum\limits_{k=1}^{N-1} \phi_t^{N,k}\phi_t^{N,N-k}$} 
  \end{subfigure} 
  \caption{As $N$ increases, the blue line $\phi_t^{N,0}\to 1$, the orange line $\phi_t^{N,1}\to -\frac{1}{2}$, and $\phi_t^{N,k}\to 0$ for $\geq 2$ in (a)-(b). $\sum\limits_{k=1}^{N-1} \phi_t^{N,k}\phi_t^{N,N-k}$ for different values of $N$ in (c).}
  \label{fig:finitephi}
\end{figure}


\section{Directed Infinite Tree Game}\label{section-tree-model}

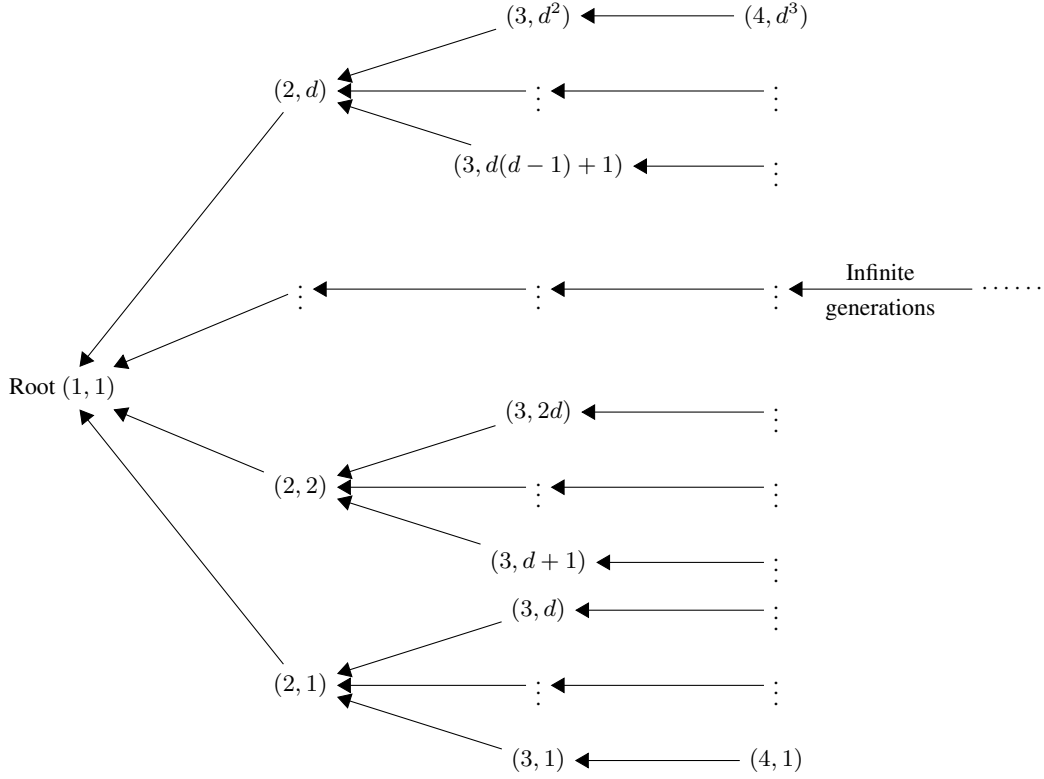
\begin{figure}[!h]
		\centering
		\begin{tikzpicture}[>={triangle 60},grow= right, 
		                                sibling distance = 7.5em,
    									level distance = 9em,
    									edge from parent/.style = {draw,{triangle 60}-}, 
    									every node/.style       = {font=\footnotesize},
  level 2/.style={sibling distance=1cm}]
  \node {Root $(1,1)$}
    child {node {$(2,1)$}
      child {node {$(3,1)$}
     	 child {node {$(4,1)$}
     	 }}
      child {node {$\vdots$}
      		child {node {$\vdots$}}}
      child {node {$(3,d)$}
      		child {node {$\vdots$}}}
   			 }
    child {node {$(2,2)$}
    child {node {$(3,d+1)$}
    			child {node {$\vdots$}}}
      child {node {$\vdots$}
     	 child {node {$\vdots$}}}
      child {node {$(3,2d)$}
      		child {node {$\vdots$}}}
    			}
    	child {node {$\vdots$}
    		child {node {$\vdots$}
    			child {node {$\vdots$}
    			child {node {$\cdots\cdots$}
    			edge from parent node [above] {Infinite}
                           node [below] {generations }}
    			}}
    	}
    	child {node {$(2,d)$}
    	child {node {$(3,d(d-1)+1)$}
    			child {node {$\vdots$}}}
      child {node {$\vdots$}
      		child {node {$\vdots$}}}
      child {node {$(3,d^2)$}
      	child {node {$(4,d^3)$}}
      		}
    	};
\end{tikzpicture}
	\caption{Directed Tree Network}
	\label{fig:tree}
\end{figure}

We describe a stochastic game on a directed tree structure with $N \ge 2$ generations first. 
Starting with one player in the root node denoted by $(1,1)$ in the first generation, recursively each parent has a fixed, common number of descendants, denoted by $d\,\geq\, 1$, and there are $d^{n-1}$ players in the $n$-th generation for $n \ge 1$. 
For $1\leq n\leq N, 1\leq k\leq d^{n-1}$, $X^{n,k}$ represents the state of the $k$-th individual of the $n$-th generation, and its direct descendants in the $(n+1)$st generation are labelled as $\{X^{n+1,(k-1)d+1},X^{n+1,(k-1)d+2},\cdots,X^{n+1,kd}\}$. We consider the stochastic differential game of players in the $N$ generations and then we generalize to a stochastic differential game in a directed infinite tree by considering its limit as $N \to \infty$. The network is shown in Figure \ref{fig:tree}.

We assume the dynamics of the states of the players are given by the stochastic differential equations of the form:
\begin{equation}\label{dynamic-tree}
    {\mathrm d}X_t^{n,k}=\a_t^{n,k}  {\mathrm d}t+\s  {\mathrm d}W_t^{n,k}, \quad 0\leq t\leq T,
\end{equation}
where $(W_t^{n,k})_{0\leq t\leq T},1\leq n\leq N,1\leq k\leq d^{n-1}$ are one-dimensional independent standard Brownian motions. Similarly, we assume that the diffusion is one-dimensional and the diffusion coefficients are constant and identical denoted by $\sigma>0$. The drift coefficients $\alpha^{n,k}$'s are adapted to the filtration of the Brownian motions and satisfy $\mathbbm{E}[\int_0^T |\alpha_t^{n,k}|^2 dt]<\infty$. 
The system starts at time $t = 0$ from $i.i.d.$ square-integrable random variables $X_0^{n,k} = \xi_{n,k}$ independent of the Brownian motions and, without loss of generality,  we assume ${\mathbbm E}(\xi_{n,k}) = 0$ for every pair of $(n,k)$. 

In this model, among the first $N-1$ generations, each player $({n,k})$ chooses its own strategy $\alpha^{n,k}$ in order to minimize its objective function of the form: for $ 1\leq n< N $
\begin{equation}\label{cost-tree}
\begin{split} 
& J^{n,k}(\a^{m,l};1\leq m \leq N, 1\leq l \leq d^m)  \\
& =\mathbbm{E} \bigg\{ \displaystyle \int_0^T \Big(\frac{1}{2}(\a^{n,k}_t)^2 
+\frac{\e}{2} \Big (\frac{1}{d}\sum_{i=(k-1)d+1}^{kd} X_t^{n+1,i}-X_t^{n,k} \Big)^2\Big) {\mathrm d}t 
+\frac{c}{2}\Big(\frac{1}{d}\sum_{i=(k-1)d+1}^{kd} X_T^{n+1,i}-X_T^{n,k} \Big)^2\bigg\},
\end{split} 
\end{equation}
for some constants $\e>0$ and $c\geq 0$. The running cost and the terminal cost functions are defined by $f^{n,k}(x,\a^{n,k})=\frac{1}{2}(\a^{n,k})^2+\frac{\e}{2}(\frac{1}{d}\sum_{i=(k-1)d+1}^{kd} x^{n+1,i}-x^{n,k})^2$ and $g^{n,k}(x)=\frac{c}{2}(\frac{1}{d}\sum_{i=(k-1)d+1}^{kd} x^{n+1,i}-x^{n,k})^2$, respectively. For simplicity, the behaviours of the $N$-th generation are described by the boundary condition where all the players $\{X^{N,k},1\leq k\leq d^{N-1}\}$ are attracted to $0$. The cost functional for player $(N,k)$ is given by:
\begin{equation*}
\quad J^{N,k}(\a^{N,k})=\mathbbm{E} \bigg\{ \displaystyle \int_0^T \big(\frac{1}{2}(\a^{N,k}_t)^2+\frac{\e}{2}(X_t^{N,k})^2\big) {\mathrm d}t +\frac{c}{2}(X_T^{N,k})^2\bigg\} 
\end{equation*}
for $k = 1, \ldots , d^{N-1}$. Since players of the last generation do not  depend on the other players, the boundary condition defines a self-controlled problem for the last generation. 

Now, inspired by the conclusion in Section \ref{section 2}, as the number $N$ of generations goes to infinity, i.e., $N \to \infty$, the effect of the boundary condition should vanish. Thus it is natural and reasonable that we decide to pass the $N$-generation finite tree to an infinite tree with infinite number of generations, and study the Nash equilibrium of the infinite-tree game. We still assume each parent has $d$ descendants. The dynamics of the states of players and the cost  functions are the same as (\ref{dynamic-tree}) and (\ref{cost-tree}) with $n\geq 1 $.

\subsection{Open-Loop Nash Equilibria}
We search for an open-loop Nash equilibrium of the directed infinite-tree system among strategies $\{\a^{n,k};n\geq 1,1\leq k\leq d^{n-1}\}$.
The Hamiltonian for player $(n,k)$ is of the form:
\begin{equation*}
\begin{split} 
& H^{n,k}(x^{m,l},y^{n,k;m,l},\a^{m,l};m\in \mathbb N,1\leq l\leq d^{m-1}) \\
& =\sum_{m=1}^{M_n}\,\sum\limits_{l=1}^{d^{m-1}} \a^{m,l} y^{n,k;m,l}+\frac{1}{2}(\a^{n,k})^2+\frac{\e}{2}  \Big(\frac{1}{d}\sum_{i=(k-1)d+1}^{kd} x^{n+1,i}-x^{n,k} \Big)^2,
\end{split} 
\end{equation*}
assuming it is defined on $Y^{n,k}_t$'s where {\it only finitely many}  $Y^{n,k;m,l}_t$'s are non-zero for every given $(n,k)$. Here, $M_n$ represents a depth of this finite dependence, a finite number depending on $n$ with $M_n>n$ for $n \ge 1$. This assumption is checked in Remark \ref{finiteYs-tree} below. Thus, the Hamiltonian $H^{n,k}$ for player $(n,k)$ is well defined for $n \ge 1$.

The adjoint processes $Y_t^{n,k}=(Y_t^{n,k;m,l};m\in \mathbb N, 1\leq l\leq d^{m-1})$ and $Z_t^{n,k}=(Z_t^{n,k;m,l;p,q};m,p\in \mathbb N, 1\leq l\leq d^{m-1},1\leq q\leq d^{p-1})$ for $n\in \mathbb N,1\leq k\leq d^{n-1}$ are defined as the solutions of the backward stochastic differential equations (BSDEs):
\begin{equation} \label{eq: last BSDEs}
    \left\{
  \begin{array}{ll}
     {\mathrm d}Y_t^{n,k;m,l}&=-\partial_{x^{m,l}}H^{n,k} (X_t,Y_t^{n,k},\a_t){\mathrm d}t+ \displaystyle \sum\limits_{p=1}^\infty \sum\limits_{q=1}^{d^{p-1}} Z_t^{n,k;m,l;p,q}{\mathrm d}W_t^{p,q}\\
    &=-\e\bigg[  \displaystyle \Big(\frac{1}{d}\sum_{i=(k-1)d+1}^{kd} X_t^{n+1,i}-X_t^{n,k} \Big)\Big(\frac{1}{d}\sum_{i=(k-1)d+1}^{kd}  \d_{(n+1,i),(m,l)}-\d_{(n,k),(m,l)} \Big)\bigg]{\mathrm d}t \\
    &\quad+ \displaystyle \sum\limits_{p=1}^\infty \sum\limits_{q=1}^{d^{p-1}} Z_t^{n,k;m,l;p,q}{\mathrm d}W_t^{p,q},\\
    Y_T^{n,k;m,l}&=\partial_{x^{m,l}}g_{n,k}(X_T) \\
    & = c\cdot  \Big(\displaystyle \frac{1}{d}\sum_{i=(k-1)d+1}^{kd} X_T^{n+1,i}-X_T^{n,k}\Big)\Big(\frac{1}{d}\sum_{i=(k-1)d+1}^{kd} \d_{(n+1,i),(m,l)}-\d_{(n,k),(m,l)}\Big).
  \end{array}
\right.
\end{equation}

\begin{remark}\label{finiteYs-tree}
For every $(m,l)\neq (n,k)$ or $(n+1,i)$ where $(k-1)d+1\leq i\leq kd$, $ {\mathrm d}Y_t^{n,k;m,l}=\sum_{p=1}^\infty \sum_{q=1}^{d^{p-1}} Z_t^{n,k;m,l;p,q}{\mathrm d}W_t^{p,q}$ and $Y_T^{n,k;m,l}=0$ implies $Z_t^{n,k;m,l;p,q}=0$ for all $(p,q)$. Thus, there must be finitely many non-zero $Y^{n,k;m,l}$'s for every $(n,k)$. Hence, the Hamiltonian can be rewritten as 
\begin{equation*}
\begin{split} 
&H^{n,k}(x^{m,l},y^{n,k;n,k},y^{n,k;n+1,i},\a^{m,l};m\in N,1\leq l\leq d^{m-1},(k-1)d+1\leq i \leq kd)\\
&= \a^{n,k} y^{n,k;n,k}+ \sum_{i=(k-1)d+1}^{kd} \a^{n+1,i} y^{n,k;n+1,i}+\frac{1}{2}\Big(\a^{n,k})^2+\frac{\e}{2} (\frac{1}{d}\sum_{i=(k-1)d+1}^{kd} x^{n+1,i}-x^{n,k}\Big)^2.
\end{split} 
\end{equation*}
\end{remark}

When $(m,l)=(n,k)$, \eqref{eq: last BSDEs} becomes:
\begin{equation} \label{BSDE-tree}
 \left\{
    \begin{array}{ll}
        {\mathrm d}Y_t^{n,k;n,k}&=\e  \Big (\displaystyle \frac{1}{d}  \sum_{i=(k-1)d+1}^{kd}X_t^{n+1,i}-X_t^{n,k} \Big){\mathrm d}t +\sum_{p=1}^\infty \sum_{q=1}^{d^{p-1}} Z_t^{n,k;n,k;p,q}{\mathrm d}W_t^{p,q},\\
    Y_T^{n,k;n,k}&=-c \Big ( \displaystyle \frac{1}{d}\sum_{i=(k-1)d+1}^{kd}X_T^{n+1,i}-X_T^{n,k} \Big).
    \end{array}
    \right.
\end{equation}

By minimizing the Hamiltonian with respect to $\a^{n,k}$, we can get an open-loop Nash equilibrium:  $\hat{\a}^{n,k}=-y^{n,k;n,k}$ for all $(n,k)$. Considering the BSDE system, we make the ansatz of the form:
\begin{equation}\label{ansatz-tree}
 Y_t^{n,k;n,k}=\sum\limits_{i=0}^\infty\phi_t^i\sum\limits_{j=0}^{d^i-1}X_t^{n+i,d^i k-j} =\sum\limits_{m=n}^\infty\phi_t^{m-n}\sum\limits_{j=0}^{d^{m-n}-1}X_t^{m,d^{m-n}k-j},
\end{equation}
for some deterministic scalar function $\phi_t$ satisfying the terminal conditions: $\phi_T^0=c,\phi_T^1=-\frac{c}{d},\phi_T^k=0$ for $k\geq 2$.
Using the ansatz, the optimal strategy $\hat{\alpha}^{n,k}$  and the forward equation for $X_{\cdot}^{n,k}$ in (\ref{dynamic-tree}) become:
\begin{equation}\label{particle_system_tree}
 \left\{
  \begin{array}{ll}
    \hat{\a}_t^{n,k}&=-Y_t^{n,k;n,k}=- \displaystyle \sum_{m=n}^\infty\phi_t^{m-n}\sum_{j=0}^{d^{m-n}-1}X_t^{m,d^{m-n}k-j},\\
     {\mathrm d}X_t^{n,k}&=- \displaystyle  \sum_{m=n}^\infty\phi_t^{m-n} \displaystyle \sum_{j=0}^{d^{m-n}-1}X_t^{m,d^{m-n}k-j} {\mathrm d}t+\s  {\mathrm d}W_t^{n,k},
  \end{array}
  \right.
\end{equation}
which gives: for $0 \le t \le T$
\[
{\mathrm d}X_t^{r,d^{r-n}k-s}= -\sum\limits_{p=r}^\infty\phi_t^{p-r}\sum\limits_{j=0}^{d^{p-r}-1}X_t^{p,d^{p-n} k-d^{p-r}s-j} {\mathrm d}t+\s {\mathrm d}W_t^{r,d^{r-n}k-s}.
\]

Differentiating the ansatz (\ref{ansatz-tree}) and substituting \eqref{particle_system_tree}, we obtain:
\begin{equation}\label{decomp-tree}
 \begin{split}
    {\mathrm d}Y_t^{n,k;n,k}&=\sum\limits_{r=n}^\infty\dot{\phi}_t^{r-n}\sum\limits_{s=0}^{d^{r-n}-1}X_t^{r,d^{r-n}k-s}{\mathrm d}t+ \sum\limits_{r=n}^\infty\phi_t^{r-n}\sum\limits_{s=0}^{d^{r-n}-1}{\mathrm d}X_t^{r,d^{r-n}k-s}\\
    &\stackrel{def}{=} \Romannum{1}\,{\mathrm d}t-\Romannum{2}\,{\mathrm d}t+\s \sum\limits_{r=n}^\infty\phi_t^{r-n}\sum\limits_{s=0}^{d^{r-n}-1} {\mathrm d}W_t^{r,d^{r-n}k-s}, 
 \end{split}
\end{equation}
where the drift term consists of two terms. First,
\begin{equation*}
  \Romannum{1} =\sum_{r=n}^\infty\dot{\phi}_t^{r-n}\sum_{s=0}^{d^{r-n}-1}X_t^{r,d^{r-n}k-s}=\sum_{r=0}^\infty\dot{\phi}_t^{r}\sum_{s=0}^{d^{r}-1}X_t^{n+r,d^{r}k-s}, 
\end{equation*}
and then, the second term is 
\begin{equation*}
\begin{split}
 \Romannum{2} 
 =\sum_{r=n}^\infty\phi_t^{r-n}\sum_{s=0}^{d^{r-n}-1} \sum_{p=r}^\infty\phi_t^{p-r}\sum_{j=0}^{d^{p-r}-1}X_t^{p,d^{p-n} k-d^{p-r}s-j}  
        = \sum_{r=0}^\infty \sum_{i=0}^{r}\phi_t^{i} \phi_t^{r-i}\sum_{s=0}^{d^{i}-1}\sum_{j=0}^{d^{r-i}-1}X_t^{n+r,d^{r} k-d^{r-i}s-j}.
\end{split}
\end{equation*}
for $0 \le t \le T$. 
Thus, (\ref{decomp-tree}) can be rewritten as:
\begin{equation}\label{ito-tree}
 \begin{split}
         {\mathrm d}Y_t^{n,k;n,k}&=\Romannum{1}\,{\mathrm d}t-\Romannum{2}\,{\mathrm d}t+\s \sum\limits_{r=n}^\infty\phi_t^{r-n}\sum\limits_{s=0}^{d^{r-n}-1} {\mathrm d}W_t^{r,d^{r-n}k-s}\\
        &= \sum\limits_{r=0}^\infty\dot{\phi}_t^{r}\sum\limits_{s=0}^{d^{r}-1}X_t^{n+r,d^{r}k-s}{\mathrm d}t-\sum\limits_{r=0}^\infty \sum\limits_{i=0}^{r}\phi_t^{i} \phi_t^{r-i}\sum\limits_{s=0}^{d^{i}-1}\sum\limits_{j=0}^{d^{r-i}-1}X_t^{n+r,d^{r} k-d^{r-i}s-j}{\mathrm d}t\\
        & \quad {} +\s \sum\limits_{r=n}^\infty\phi_t^{r-n}\sum\limits_{s=0}^{d^{r-n}-1} {\mathrm d}W_t^{r,d^{r-n}k-s}.
\end{split}
\end{equation}

Now comparing the two It\^o's decompositions (\ref{BSDE-tree}) and (\ref{ito-tree}), we obtain first the processes 
$Z_t^{n,k;n,k;p,q}$ from the martingale terms :
\begin{equation*}
   Z_t^{n,k;n,k;p,q}=\s\phi_t^{p-n} \text{ for } p\geq n \text{ and } 1\leq q\leq d^{p-n}\, ;\quad Z_t^{n,k;n,k;p,q}=0,  \text{ otherwise}, 
\end{equation*}
and we obtain second from the drift terms: 
\begin{equation} \label{riccati_tree}
\begin{array}{rll}
     \text{for}\ k=0:&\Dot{\phi}_t^0=\phi_t^0\cdot \phi_t^0-\e , &\phi_T^0=c, \\
     \text{for}\ k=1:& \Dot{\phi}_t^1=2\phi_t^0\cdot \phi_t^1+\frac{\e}{d} , &\phi_T^1=-\frac{c}{d},\\
\text{for}\ k\geq 2:& \Dot{\phi}_t^k=\phi_t^0\cdot \phi_t^k+\phi_t^1\cdot\phi_t^{k-1}+\cdots+\phi_t^{k-1}\cdot\phi_t^1+\phi_t^k\cdot\phi_t^0 ,&  \phi_T^k=0.
\end{array}
\end{equation}

This Riccati system is closely related to the one in (\ref{eq18}) for the infinite-player directed chain game and we can have a similar lemma.

\begin{Lemma}\label{inf_sumo_tree}
Let $\phi_t^{(k)}=\phi_t^{k}$ in (\ref{riccati_tree}) to avoid confusion. We have $\sum_{k=0}^{\infty} d^k \phi_t^{k}=0$, and the functions $\phi^k$'s can be obtained by a series expansion.
\end{Lemma}
\begin{proof} Given in Appendix \ref{appendix9}.\end{proof}

\subsection{Catalan Markov Chain for the Directed Tree Model}
Without loss of generality, we assume $\e=1$ and $\sigma=1$.  Following section \ref{section 4}, by taking $T\to \infty$, we look at the stationary long-time behavior of the Riccati system (\ref{riccati_tree}) satisfying $\dot{\phi}^k=0$ for all $k$. Then the system gives the recurrence relation: 
\[
\phi^0=1, \quad \phi^1=-\frac{1}{2d}, \quad  \text{ and } \quad \sum\limits_{k=0}^n \phi^k\phi^{n-k}=0. 
\]
By using a moment generating function method as in Appendix \ref{moment_generating_method_tree}, we obtain the stationary solution (cf. \eqref{eq: phi01j}):
\[
\phi^0=1,\quad \phi^1=-\frac{1}{2d} ,\quad  \text{ and } \quad \phi^k=-\frac{(2k-3)!}{(k-2)!k!2^{2k-2}d^k}  \quad \text{ for } \quad k\geq 2. 
\]

Let $q_0=-\phi^0=-1, q_1=-d\phi^1=\frac{1}{2} $,  and $q_k=-d^k\phi^k=\frac{(2k-3)!}{(k-2)!k!}\,\frac{1}{2^{2k-2}} $ for $k\geq 2$.
We consider the continuous-time Markov chain with state space $\, \mathbb N \,$ and generator matrix: 
$$\mathbf{Q}_{d\text{-tree}}\,=-\left( \begin{array}{cccccc} 
\phi^0 & d\phi^1 &  \cdots&   d^k\phi^k & \cdots& \cdots \\
0 & \phi^0 & d\phi^1 &  \cdots &   d^k\phi^k& \ddots \\
0 &0 & \phi^0& \ddots& \ddots& \ddots\\
& \ddots & \ddots &\ddots  & \ddots& \ddots\\
\end{array} 
\right)
\,= \left( \begin{array}{ccccc} 
q_0 & q_1 &  \cdots & q_k & \cdots \\
0 & q_0 & q_1&  \ddots & \ddots \\
0 &0 & q_0 & \ddots &\ddots\\
& \ddots & \ddots &\ddots  & \ddots\\
\end{array} 
\right)
= \mathbf{Q}
,
$$ 
where $\mathbf{Q}$ is the generator matrix \eqref{eq: Q:sec4} of the continuous-time Markov chain for the  infinite-player directed chain game in section \ref{section 4}. 

Thus, as $T \to \infty$, the limit of the infinite particle system (\ref{particle_system_tree}) can be rewritten as a linear stochastic evolution equation of Ornstein-Uhlenbeck type: 
\begin{equation}  \label{evoution_eqn}
{\mathrm d}\mathbf{X}_t=\mathbf{Q_{d\text{-tree}}\,X}_t {\mathrm d}t+ {\mathrm d}\mathbf{W}_t,  \quad t \ge 0 
\end{equation}
where $\mathbf{X_.}=(\overline{X}_t^{k} = \sum_{i=1}^{d^{k-1}}X_{.}^{k,i}/d^{k-1},k\in \mathbb{N})$ with $\mathbf{X}_0=\mathbf{x}_0$ and $\mathbf{W_.}=( \overline{W}_t^{k} = \sum_{i=1}^{d^{k-1}}W_{.}^{k,i} / d^{k-1},k\in \mathbb{N})$ is a vector of averaged Brownian motions with mean $0$ and variance $t / d^{k} $ in each generation $k \in \mathbb N$. Its solution is formally given by 
\begin{equation}
\mathbf{X}_t=e^{t\mathbf{Q}_{d\text{-tree}}}\mathbf{x_0}+ \int_0^t e^{(t-s)\mathbf{Q}_{d\text{-tree}}}{\mathrm d}\mathbf{W}_s ;\quad t\geq 0.
\end{equation}


Similar to proposition \ref{sol_markovc} and appendix \ref{appendix3} and \ref{appendixvariance}, we can find the formula for $\overline{X}_t^{1} $ and its asymptotic variance. Proof is shown Appendix \ref{solution_X^1bar}. Since by definition: $X_t^{1,1}=\overline{X}_t^{1}$, then we have the following result  by the Markov chain method for the directed tree model:
\begin{prop}
The formula for the root node $X_t^{1,1}$ in \eqref{evoution_eqn} is:
\[
X_t^{1,1}\, =\, \sum\limits_{j=1}^{\infty}\int^{t}_{0}  \frac{\,(t-s)^{2(j-1)}  \,}{\, (j-1)!\,} \cdot \rho_{j-1}(- (t-s)^{2}) \, e^{-(t-s)} \cdot  {\mathrm d} \overline{W}_{s}^{j} , \quad t \ge 0 , 
\]
where $\rho_{i}(\cdot)$ is defined in \eqref{eq: rhokxProp1}. 
Moreover, the asymptotic variance of $X_t^{1,1}$ is finite, i.e.   
\[
\lim\limits_{t \to \infty}\text{Var} ( X_{t}^{1,1})\,=\,\frac{\,\sqrt{2}\,}{\,2\,}\cdot \dfrac{1}{\sqrt{1+\sqrt{\frac{d-1}{d}}}} \in \Big(\frac{1}{2}, \frac{\,\sqrt{2}\,}{\,2\,} \Big].
\]
\end{prop}

\begin{remark}
When $d$ goes to infinity, we are in the regime of the mean field game. The asymptotic variance is $\frac{1}{2}$ which is consistent with the variance of an Ornstein--Uhlenbeck process where the particle is attracted to $0$ and the volatility and the mean reversion constant are both $1$.
\end{remark}

\section{Conclusion} \label{section 8}

We studied a linear-quadratic stochastic differential game on a directed chain network. We were able to identify Nash equilibria in the case of finite chain with various boundary conditions and in the case of an infinite chain. This last case allows for more explicit computation in terms of Catalan functions and Catalan Markov chain. The Catalan open-loop Nash equilibrium that we obtained is characterized by interactions with all the neighbors in one direction of the chain weighted by Catalan functions, event though the interaction in the objective functions is only with the nearest neighbor. Under equilibrium the variance of a state converges in the infinite time limit as opposed to the diverging behavior observed in the nearest neighbor dynamics studied in Detering, Fouque \& Ichiba \cite{Nils-JP-Ichiba2018DirectedChain}. Our analysis is extended to mixed games with directed chain and mean field interaction so that our game model includes the two extreme network interactions, fully connected and only one neighbor connection. It is also extended to game on a deterministic tree structure. Our ongoing and future research concerns games with interactions on directed tree-like  stochastic networks modeled as branching processes.


\newpage

\appendix
\section{Appendix}\label{Appendix}
\subsection{Proof of Lemma \ref{inf_sumo} in Section \ref{section 3}}\label{appendix1}
Define the generating function $S_t(z)=\sum_{k=0}^\infty z^k\ \phi_{t}^{(k)}$ where $0\leq z< 1$ with $\phi_{t}^{(k)}=\phi_{t}^{k}$ in (\ref{eq18}) to avoid confusion. Then substituting (\ref{eq18}), we obtain 
\begin{equation} \label{S_t}
\begin{split} 
\Dot{S}_t(z) &= \sum\limits_{k=0}^\infty z^k\Dot{\phi}_{t}^{(k)}\\
&= \big(\phi_{t}^{(0)}\phi_{t}^{(0)}-\epsilon\big)+z \big( \phi_{t}^{(0)}\phi_{t}^{(1)}+\phi_{t}^{(1)}\phi_{t}^{(0)}+\epsilon \big)+\cdots\\
&+z^k\,\big( \phi_{t}^{(0)}\phi_{t}^{(k)}+\phi_{t}^{(1)}\phi_{t}^{(k-1)}+\cdots +\phi_{t}^{(k-1)}\phi_{t}^{(1)}+\phi_{t}^{(k)}\phi_{t}^{(0)} \big)+\cdots\\
&=\bigg( \phi_{t}^{(0)}S_t(z)+z\phi_{t}^{(1)}S_t(z)+\cdots+z^k\phi_{t}^{(k)}S_t(z)+\cdots\bigg)-\epsilon+z\epsilon\\
&=(S_t(z))^2-\epsilon(1-z),\\
S_T(z)&=c(1-z).
\end{split} 
\end{equation}

\noindent $\,\bullet\,$ 
For $z=1$, we get the ODE: 
\begin{equation}\label{S ODE}
\Dot{S}_t(1)=(S_t(1))^2\, , \quad \quad  S_T(1)=0.
\end{equation}
The solution is $S_t(1)=0$, and we deduce:
\begin{equation*}
\sum\limits_{k=0}^\infty  \phi_{t}^{(k)}=0,\quad i.e.,\quad \phi_{t}^{(0)}=-\sum\limits_{k=1}^\infty  \phi_{t}^{(k)}.
\end{equation*}

One needs to be careful when taking $z=1$ because the series defining $S_t(1)$ may not converge. Instead, we take a sequence $\{z_n\}$ converging to $1$, the limit of $S_t(z_n)$ converges to the ODE (\ref{S ODE}), and we get the conclusion.

\noindent $\,\bullet\,$  For $z\neq 1$, the solution to the Riccati equation (\ref{S_t}) satisfies 
\begin{equation}\label{sol_St}
     \begin{split} 
     S_t(z)& = \dfrac{-\e(1-z)\big(e^{2\sqrt{\e(1-z)}(T-t)}-1\big)-c(1-z)\big(\sqrt{\e(1-z)}e^{2\sqrt{\e(1-z)}(T-t)}+\sqrt{\e(1-z)}\big)}{\big(-\sqrt{\e(1-z)}e^{2\sqrt{\e(1-z)}(T-t)}-\sqrt{\e(1-z)}\big)-c(1-z)\big(e^{2\sqrt{\e(1-z)}(T-t)}-1\big)}\\
     &=\dfrac{\big(-\e(1-z)-c\sqrt{\e(1-z)}(1-z)\big)e^{2\sqrt{\e(1-z)}(T-t)}+\e(1-z)-c\sqrt{\e(1-z)}(1-z)}{\big(-\sqrt{\e(1-z)}-c(1-z)\big)e^{2\sqrt{\e(1-z)}(T-t)}-\sqrt{\e(1-z)}+c(1-z)}\\
     &\xrightarrow[T\to \infty]{} 
     \sqrt{\e(1-z)}.
     \end{split} 
\end{equation}

\subsection{Catalan Markov Chain and Proposition \ref{sol_markovc} in Section \ref{section 4}} \label{CatalanMC-1}
We have the Catalan probabilities: $\sum\limits_{k=1}^\infty p_k=1$ and $p_k=\frac{1}{2}\, \sum\limits_{i=1}^{k-1}p_i p_{k-i}$. Then, it is easily seen that $\, \mathbf{Q}^{2} \, =\,  I - B\,$ with $\, B\,$ having $\,1\,$'s on the upper second diagonal and $\,0\,$'s elsewhere, i.e., 
\[
\mathbf{Q}^{2} \, =\, \left( \begin{array}{ccccc} 
1 & -1 &  0 & \cdots & \\
0 & 1 & -1 &  \ddots &  \\
& \ddots & \ddots &\ddots  & \\
\end{array} 
\right) \, = \, - J_{\infty} (-1)  \, , \quad J_{\infty} (\lambda) \, :=\,  \left( \begin{array}{ccccc} 
\lambda & 1 &  0 & \cdots & \\
0 & \lambda & 1 &  \ddots &  \\
& \ddots & \ddots &\ddots  & \\
\end{array} 
\right) \, . 
\]
Here, $\, J_{\infty}(\lambda) \,$ is the infinite Jordan block matrix with diagonal components $\,\lambda\,$.

The matrix exponential of $\, \mathbf{Q} t\,$, $\, t \ge 0 \,$, is written formally as
\[
\exp ( \mathbf{Q} t ) \, =\,  F(-\mathbf{Q}^{2} t^{2}) \, =\,  F ( J_{\infty} (-1) \cdot t^{2} ) \, , \, \, t \ge 0 \, ,  \quad F(x) \, :=\, \exp ( - \sqrt{ - x } )\, ,  \, \, x \in \mathbb C \, . 
\]
Since a smooth function of a Jordan block matrix can be expressed as 
\[
F(J_{\infty} (\lambda) ) \, =\, F(\lambda I+B)=\,\sum\limits_{k=0}^\infty \frac{F^{(k)}(\lambda)}{k!}\, B^k =\, \left( \begin{array}{cccccc} 
F(\lambda) & F^{(1)}(\lambda) & \frac{\,F^{(2)}(\lambda) \,}{\,2!\,} & \cdots & \frac{\,F^{(k)}(\lambda) \,}{\,k!\,} & \cdots \\
& \ddots & \ddots & \ddots & & \ddots \\
 & & \ddots & \ddots & \ddots & \\ 
\end{array} \right ) \, ,
\]
we get
\[
\exp (\mathbf{Q} t )=\, F(J(-\infty)\cdot t^2)=\, F((-I+B)t^2)=\,\sum\limits_{k=0}^\infty \frac{F^{(k)}(-t^2)}{k!}\, (Bt^2)^k=\,\sum\limits_{k=0}^\infty \frac{t^{2k}F^{(k)}(-t^2)}{k!}\, B^k  .
\]
The $\,(j,k)\,$-element of $\, \exp ( \mathbf{Q} t ) \,$ is formally given by 
\[
(\exp (\mathbf{Q} t ))_{j,k} \, =\,  \frac{\,t^{2(k-j)} \cdot F^{(k-j)}(-t^{2}) \,}{\, (k-j)!\,} \,  , \quad j \le k \, , \, \, \text{ where } \, \, F^{(k)} (x) \, :=\, \frac{\,{\mathrm d}^{k}  F \,}{\,{\mathrm d} x^{k}\,}(x)  \, ; \quad x > 0 \,, \, \, k \in \mathbb N \,  ,   
\]
and $\, (\exp ( \mathbf{Q} t))_{j,k} \, =\,  0 \,$, $\, j > k \,$ for $\, t \ge 0 \,$. Here the $\,k\,$-th derivative $\, F^{(k)}(x) \,$ of $\, F(\cdot) \,$ can be written as $\,F^{(k)}(x) \, =\,  \rho_{k}(x) e^{ - \sqrt{-x}}\,$, where $\, \rho_{k}(x) \,$ satisfies the recursive equation 
\[
\rho_{k+1}(x) \, =\, \rho^{\prime}_{k}(x) +\frac{\,\rho_{k}(x)\,}{\,2 \sqrt{ - x} \,} \, ; \quad k \ge 0 \, , \,  
\]
with $\, \rho_{0}(x) \, =\,  1 \, $, $\, x \in \mathbb C \,$. For example, 
\[
\rho_{0}(x) \, =\,  1\,,\quad \rho_{1}(x) \, :=\, \frac{+1}{2} (-x)^{-\frac{1}{2}}\,,\quad \rho_{2}(x) \, :=\, \frac{1}{4} (-x)^{-\frac{2}{2}}+\frac{+1}{4}(-x)^{-\frac{3}{2}},
\]
\[
\rho_{3}(x) \, :=\, \frac{1}{8} (-x)^{-\frac{3}{2}}+\frac{3}{8} (-x)^{-\frac{4}{2}}+\frac{3}{8} (-x)^{-\frac{5}{2}},
\]
\[
\rho_{4}(x) \, :=\, \frac{1}{16} (-x)^{-\frac{4}{2}}+\frac{6}{16} (-x)^{-\frac{5}{2}}+\frac{15}{16} (-x)^{-\frac{6}{2}}+\frac{15}{16} (-x)^{-\frac{7}{2}},
\]
\[
\rho_{5}(x) \, :=\, \frac{1}{32} (-x)^{-\frac{5}{2}}+\frac{10}{32} (-x)^{-\frac{6}{2}}+\frac{45}{32} (-x)^{-\frac{7}{2}}+\frac{105}{32} (-x)^{-\frac{8}{2}}+\frac{105}{32} (-x)^{-\frac{9}{2}}.
\]
More generally, by mathematical induction, we may verify 
\begin{eqnarray}
\rho_k(x) 
                &=&\frac{1}{2^k} \sum\limits_{j=k}^{2k-1}\, \frac{(j-1)!}{(2j-2k)!!(2k-j-1)!} \, (-x)^{\,-\frac{j}{2}}, \quad 
                k\geq 1. \ \label{rhok}
\end{eqnarray}
Therefore, substituting them into \eqref{solution_mc}, we obtain Proposition \ref{sol_markovc}. 

%

\subsection{Proof of Remark \ref{rho_k and var} in Section \ref{section 4}}\label{appendix3}
By $\rho_k$'s formulae in (\ref{rhok}), we have for $\nu \ge 0 $, $ k \ge 1$, 
\begin{equation*}
\begin{split}
\rho_k(-\nu^2)=&\frac{1}{2^k} \sum\limits_{j=k}^{2k-1}\,  \frac{(j-1)!}{(2j-2k)!!(2k-j-1)!} \,
\, =\,  \frac{\,1\,}{\,2^{k} \nu^{k}\,}\cdot \sqrt{\frac{\,2 \nu \,}{\,\pi \,} } \cdot  e^{\nu} \cdot K_{k-(1/2)} (\nu ) , 
\end{split}
\end{equation*}
 where $\, K_{n}(x)\,$ is the modified Bessel function of the second kind, i.e., 
\[
K_{n}(x) \, =\,  \int^{\infty}_{0} e^{-x \cosh t} \cosh (nt ) {\mathrm d} t \, ; \quad n > -1 , \,\,\, x > 0 \, . 
\]
Then, by the change of variables, we obtain 
\[
\text{Var} ( X_{0} (t))\, =\,  \sum\limits_{k=0}^{\infty} \int^{t}_{0} \frac{\,(t-s)^{4k}\,}{\,(k!)^{2} \,} \lvert  \rho_{k} (- (t-s)^{2}) \rvert^{2} e^{-2(t-s)}{\mathrm d} s 
\]
\[
{\, =\, \sum_{k=1}^{\infty}  \int^{t}_{0} \frac{\,2\,}{\,\pi\,}  \frac{\,\nu^{2k+1}\,}{\,(k!)^{2}\, 4^{k}\,} \big( K _{k-(1/2)}(\nu) \big)^{2} {\mathrm d} \nu + \frac{\,1 - e^{-2t}\,}{\,2\,} \, ; \quad t \ge 0 \, . } 
\]

\subsection{Proof of  Proposition \ref{finite_limit_var} in Section \ref{section 4}}\label{appendixvariance}
Using the following identities from the special functions 
\begin{equation*}
\begin{split} 
&\int_0^\infty t^{\alpha-1}(K_\nu(t))^2 dt= \frac{\sqrt{\pi}}{4\Gamma ((\alpha+1)/2 )} \Gamma\Big(\frac{\alpha}{2}\Big)\Gamma\Big(\frac{\alpha}{2}-\nu\Big)\Gamma \Big(\frac{\alpha}{2}+\nu\Big),\\
\\
&
\frac{\sqrt{2}}{4}x \sqrt{x^2-\sqrt{x^4-16}}=
\sum\limits_{k=0}^\infty {4k \choose 2k}\frac{1}{2k+1}\frac{1}{x^{4k}},\quad \text{for } x\geq 2. \\
\end{split} 
\end{equation*}
based on Remark \ref{rho_k and var}, we obtain the limit of variance of $X_{t}^{1}$, as $\,t \to \infty\,$, i.e.,  $\lim_{t\to \infty} \text{Var} ( X_{t}^{1}) =  $ 
\begin{equation*}
\begin{split}
& \frac{\,1\,}{\,2\,} + \sum_{k=1}^{\infty} \int^{\infty}_{0} \frac{\,2\,  s^{2k+1}\,}{\,\pi (k!)^{2} 4^{k}\,} \cdot [ K_{k-(1/2)}(s)]^{2} {\mathrm d} s 
=\, \frac{\,1\,}{\,2\,} + \sum_{k=1}^{\infty} \frac{\,2\,}{\,\pi \, (k!)^{2} 4^{k}\,} \int^{\infty}_{0} s^{2k+1} [ K_{k-(1/2)}(s)]^{2}{\mathrm d} s \,\\
\, \\
&=\, \frac{\,1\,}{\,2\,} + \sum_{k=1}^{\infty} \frac{\,2 \,}{\,\pi (k!)^{2} 4^{k}\,} \cdot \frac{\, \pi  \, \Gamma ( k + 1) \,  \Gamma ( 2k + (1/2)) \,}{\, 8\,  \Gamma ( k + (3/2)) \,}
{=\, \frac{\,1\,}{\,2\,} + \frac{1}{2} \sum_{k=1}^{\infty}   {4k \choose 2k} \frac{1}{2k+1} \frac{1}{2^{4k}} \,=\, \frac{1}{2} \sum_{k=0}^{\infty}   {4k \choose 2k} \frac{1}{2k+1} \frac{1}{2^{4k}} \,}\\
\, &=\,  \frac{\,1\,}{\,2\,}\, \cdot \, \frac{\sqrt{2}}{4}2\sqrt{2^2-0}=\frac{1}{\sqrt{2}}.  
\end{split}
\end{equation*}

\subsection{Proofs of Propositions \ref{aurocov_sim}- 
 \ref{crosscov_sim} in Section \ref{section 4}}\label{appendixcovariance}
From the expression (\ref{xforcovariance}) for $X_t^{1}$, 
the auto-covariance $\mathbbm{E}[X_s^{1}X_t^{1}] $ and the cross covariance $\mathbbm{E}[X_t^{1}X_t^{j+1}]$ are  
\begin{equation} \label{eq: autocov-crosscov}
\begin{split}
\mathbbm{E}[X_s^{1}X_t^{1}] 
&=\sum\limits_{i=0}^\infty  \frac{1}{\pi (i!)^2 2^{2i-1}}  \int_0^s(t-v)^{i+1/2} (s-v)^{i+1/2} K_{i-1/2}(t-v) K_{i-1/2}(s-v) {\mathrm d} v\\
&= \sum\limits_{i=0}^\infty \frac{1}{\pi (i!)^2 2^{2i-1}} \int_0^s  ((t-s+v)v)^{i+1/2}  K_{i-1/2}(t-s+v) K_{i-1/2}(v) {\mathrm d} v >  0 ;  \\
\mathbbm{E}[X_t^{1}X_t^{j+1}]
&=\sum\limits_{i=j}^\infty \int_0^t \frac{1}{\pi i!(i-j)! } \frac{(t-\nu)^{2i-j+1}}{2^{2i-j-1}}   K_{i-1/2}(t-\nu) K_{i-j-1/2}(t-\nu) {\mathrm d} \nu\\
&= \sum\limits_{i=0}^\infty  \frac{1}{\pi (j+i)!j!} \frac{1}{2^{j+2i-1}}  \int_0^t s^{j+2i+1} K_{j+i-1/2}(s) K_{i-1/2}(s) {\mathrm d} s \\
& \xrightarrow[t\to \infty]{} \sum\limits_{i=0}^\infty  \frac{1}{\pi (j+i)!j!} \frac{1}{2^{j+2i-1}}  \int_0^\infty s^{j+2i+1} K_{j+i-1/2}(s) K_{i-1/2}(s) {\mathrm d} s > 0 \, . 
\end{split}
\end{equation}

By the Cauchy-Schwarz inequality, as  $ t\to \infty$, the asymptotic cross covariance between $X_{t}^{1}$ and $X_{t}^{j+1}$ is bounded by  
\begin{equation} \label{eq: bound cross covariance}
\, \lim_{t \to \infty}\mathbbm E [ X_{t}^{1} X_{t}^{j+1} ] \le \lim_{t\to \infty} (\mathbb E [ (X_{t}^{1})^{2} ])^{1/2} \cdot (\mathbbm E [ (X_{t}^{j+1})^{2} ])^{1/2} \, =\,  \lim_{t\to \infty} \text{Var} ( X_{t}^{1}) \, =\, \frac{\,1\,}{\,\sqrt{2} \,}  
\end{equation}
for $j \ge 0$, because $X_{\cdot}^{1}$ and $X_{\cdot}^{j+1}$ have the same distribution. 

To compute the asymptotic auto-covariance, fix $s > 0 $ and let $t \to \infty$. By the asymptotic expansion of the modified Bessel function $K_{\alpha}(z) $, $z > 0 $, there exists a positive constant $c > 0 $ such that for every sufficiently large $t  (> s) $ 
\[
\sup_{i \ge 0  } \frac{\,1\,}{\,i! \, t^{i+1}\,}\int_0^s  ((t-s+v)v)^{i+1/2}  K_{i-1/2}(t-s+v) K_{i-1/2}(v) {\mathrm d} v \le c \cdot e^{- (t-s) } \,. 
\]
Then combining this estimate with (\ref{eq: autocov-crosscov}), we obtain 
\[
\mathbbm E [ X_{s}^{1} X_{t}^{1} ] \le \sum_{i=0}^{\infty} \frac{\,4c t^{i+1} e^{-(t-s)} \,}{\,\pi \, i ! \, 4^{i}\,}  \le \frac{\,4c t \,}{\,\pi \,} e^{-(t-s) + (t/4)} \xrightarrow[t \to \infty]{} 0 . 
\]

\subsection{Proof of Proposition \ref{infinite-sum=0} in Section \ref{section 6}}\label{infinite-sum0}
Adding equations (\ref{eq36}) and (\ref{eq39}), for $0<u<1$, we get:
\begin{align*}
    u\Dot{\phi}_t^{0}&=u^2(\phi_t^0)^2+2u(1-u)\psi_t\phi_t^0-u(1-u)\psi_t \sum\limits_{k=0}^\infty \phi_t^k-u\e .
\end{align*}
Then (\ref{eq38}) and (\ref{eq39}) can be written as:
\begin{align*}
    u\Dot{\phi}_t^1=2u^2\phi_t^0\phi_t^1+2u(1-u)\psi_t\phi_t^1+u\e, 
\quad     u\Dot{\phi}_t^k=u^2\sum\limits_{j=0}^k\phi_t^j\phi_t^{k-j}+2u(1-u)\psi_t\phi_t^k, \quad\text{ for}\ k\geq2.
\end{align*}
Define $S_t(z)=\sum\limits_{k=0}^\infty z^k \phi_{t}^{(k)}$ where $0\leq z\leq 1$ and $\phi_{t}^{(k)}=\phi_{t}^{k}$ in equations above to avoid confusion. Then
\begin{equation} 
\begin{array}{ll} \label{eq40}
u\Dot{S}_t(z) &= \displaystyle \sum\limits_{k=0}^\infty z^k u\Dot{\phi}_{t}^{(k)} 
= u^2(S_t(z))^2+u(1-u)\psi_tS_t(z)-u(1-z)\e,\\
uS_T(z)&=u(1-z)c. 
\end{array}
\end{equation}
For  $z=1$, we obtain the ODE:
\begin{equation}
u\Dot{S}_t(1)=u^2(S_t(1))^2+u(1-u)\psi_tS_t(1)\, ,\quad uS_T(1)=0.
\end{equation}
The solution is given by $S_t(1)=0$ and we deduce $\sum_{k=0}^\infty \phi_t^{(k)}=0.$

\subsection{About Table \ref{table:1} in Section \ref{section 6}}\label{table1}
From Proposition \ref{catalanMC-sol-general}, for $\, t \ge 0 \,$, we have:
\begin{equation}
\begin{split}
\text{Var} ( X_{t}^{1} ) \, &=\, \text{Var} \Big(  \sum_{k=0}^{\infty} \int^{t}_{0}  \frac{\,u^k (t-s)^{2k} \,}{\,k!\,} F^{(k)}(-(t-s)^{2}) {\mathrm d} W_{k}(s) \Big)  \\
&\, =\,  \sum_{k=0}^{\infty} \int^{t}_{0} \frac{\,u^{2k} (t-s)^{4k}\,}{\,(k!)^{2} \,} \lvert  \rho_{k} (- (t-s)^{2}) \rvert^{2} e^{-2(t-s)}{\mathrm d} s\\
&=  \sum_{k=1}^{\infty} \int^{t}_{0} \frac{\,2 u^{2k} \,}{\,\pi (k!)^{2} 4^k\,}\nu^{2k+1} (K_{k-\frac{1}{2}}(\nu))^2  {\mathrm d} \nu+\frac{1-e^{-2t}}{2}.
\end{split}
\end{equation}
As $t\to \infty$, for $u<1$, we obtain 
\begin{equation*}
\begin{split}
\lim_{t\to \infty} \text{Var} (X^{1}_{t})
\, &=\,  \frac{\,1\,}{\,2\,} + \sum_{k=1}^{\infty} \int^{\infty}_{0} \frac{\,2u^{2k}\,  s^{2k+1}\,}{\,\pi (k!)^{2} 4^{k}\,} \cdot [ K_{k-(1/2)}(s)]^{2} {\mathrm d} s 
= \frac{\,1\,}{\,2\,} \Big(1-\frac{u^2}{2} \Big)^{-\frac{1}{2}}<\infty.
\end{split}
\end{equation*}

Since we have 
\begin{equation*}
\begin{split}
X_{t}^{j}&= \sum\limits_{k=j}^\infty \bigintssss_0^t \frac{u^{k-j}(t-s)^{2(k-j)}}{(k-j)!} \rho_{k-j}(- (t-s)^{2}) \, e^{-(t-s)} {\mathrm d} W_{k}(s)
= \sum\limits_{i=0}^\infty \int_0^t \frac{u^i}{\sqrt{\pi} i!}  \frac{(t-s)^{i+1/2}}{2^{i-1/2}} K_{i-1/2}(t-s) {\mathrm d} W_{j+i}(s),
\end{split}
\end{equation*}
the (auto)covariance is: 
\begin{equation*}
\begin{split}
\mathbb{E}[X_{s}^{1} X_{t}^{1}]  
&=\sum\limits_{k=0}^\infty \int_0^s \frac{u^{2k}}{\pi (k!)^2 2^{2k-1}}  (t-\nu)^{k+1/2} (s-\nu)^{k+1/2} K_{k-1/2}(t-\nu) K_{k-1/2}(s-\nu) {\mathrm d} \nu\\
&= \sum\limits_{k=0}^\infty \int_0^s \frac{u^{2k}}{\pi (k!)^2 2^{2k-1}}  ((t-s+\alpha)\alpha)^{k+1/2}  K_{k-1/2}(t-s+\alpha) K_{k-1/2}(\alpha) {\mathrm d} \alpha  \neq 0 . 
\end{split}
\end{equation*}
\bigskip
The cross-covariance is:
\begin{equation*}
\begin{split}
\mathbbm{E}[X_{t}^{1} X_{t}^{k+1}]  
&=\sum\limits_{i=k}^\infty \int_0^t \frac{u^{2i-k}}{\pi i!(i-k)!)} \frac{(t-\nu)^{2i-k+1}}{2^{2i-k-1}}   K_{i-1/2}(t-\nu) K_{i-k-1/2}(t-\nu) {\mathrm d} \nu\\
&= \sum\limits_{j=0}^\infty  \frac{u^{k+2j}}{\pi (k+j)!j!} \frac{1}{2^{k+2j-1}}  \int_0^t s^{k+2j+1} K_{k+j-1/2}(s) K_{j-1/2}(s) {\mathrm d} s\,,\\
\end{split}
\end{equation*}
and as $t \to \infty$ 
it converges to
\begin{equation}\label{crosscov-u}
 \sum\limits_{j=0}^\infty  \frac{u^{k+2j}}{\pi (k+j)!j!} \frac{1}{2^{k+2j-1}}  \int_0^\infty s^{k+2j+1} K_{k+j-1/2}(s) K_{j-1/2}(s) {\mathrm d} s \quad(\mathbf{\neq 0,\, if\, u\neq 0} ),
\end{equation} 
and as in \eqref{eq: bound cross covariance}, we deduce the asymptotic upper bound 
 \begin{equation*}
\begin{split}
&\lim\limits_{t\to\infty} \mathbbm{E}[X^{(u)}_0(t)X^{(u)}_k(t)]
<\frac{1}{2}\Big(1-\frac{u^2}{2}\Big)^{-1/2} . 
\end{split}
\end{equation*}

\subsection{Proof of Proposition \ref{finite_sumo} in Section \ref{section 7}}\label{appendixsumo}
Define $S_t^N(z)=\sum_{k=0}^{N-1} z^k\phi_t^{N,k}$, then, by  (\ref{eq10}), we have:
\begin{equation}
\begin{split}
\dot{S}_t^N(z)&=\sum\limits_{k=0}^{N-1} z^k\dot{\phi}_t^{N,k} 
=(S_t^N(z))^2+(1-z^N)\big[\sum\limits_{j=0}^{N-2}z^j \cdot \sum\limits_{k=j+1}^{N-1} \phi_t^{N,k}\phi_t^{N,N+j-k}\big]-(1-z)\e, \\
\end{split}
\end{equation}
with $S_{T}^{N}(z) = (1-z) c $. For $z=1$, $\dot{S}_t^N(1)=(S_t^N(1))^2,\, S_T^N(1)=0$, and 
\begin{equation*}
S_t^N(1)=\sum\limits_{k=0}^{N-1} \phi_t^{N,k}=0,\, i.e., \quad \phi_t^{N,0}=-\sum\limits_{k=1}^{N-1} \phi_t^{N,k}.
\end{equation*} 

\subsection{Proof of Lemma \ref{inf_sumo_tree} in Section \ref{section-tree-model} }\label{appendix9}
Similar to the proof of lemma \ref{inf_sumo} in Appendix \ref{appendix1}, define $S_t(z)=\sum_{k=0}^\infty z^k\ \psi_{t}^{(k)}$where $0\leq z< 1$ and $\psi_t^{(k)}=d^k\phi_t^{k}$ in equation (\ref{riccati_tree}). The Riccati system for $\psi$ functions is given by:
\begin{equation*}
\begin{array}{rll}
     \text{for}\ k=0:&\Dot{\psi}_t^{(0)}=\psi_t^{(0)}\cdot \psi_t^{(0)}-\e , &\psi_T^{(0)}=c, \\
     \text{for}\ k=1:& \Dot{\psi}_t^{(1)}=2\psi_t^{(0)}\cdot \psi_t^{(1)}+\e , &\psi_T^{(1)}=-c,\\
\text{for}\ k\geq 2:& \Dot{\psi}_t^{(k)}=\psi_t^{(0)}\cdot \psi_t^{(k)}+\psi_t^{(1)}\cdot\psi_t^{(k-1)}+\cdots+\psi_t^{(k-1)}\cdot\psi_t^{(1)}+\psi_t^{(k)}\cdot\psi_t^{(0)} ,&  \phi_T^{(k)}=0.
\end{array}
\end{equation*}
Then similar to equation (\ref{S_t}):
\begin{equation}\label{sum_tree}
\begin{split} 
\Dot{S}_t(z) &= \sum\limits_{k=0}^\infty {z}^k\Dot{\psi}_{t}^{(k)}=(S_t(z))^2-\epsilon(1-z), \quad  0 \le t \le T \\
\end{split} 
\end{equation}
with $S_T(z)=dc(1-z)$. 
For $z=1$, we get the same ODE as (\ref{S ODE}): 
\begin{equation}
\Dot{S}_t(1)=(S_t(1))^2\, , \quad \quad  S_T(1)=0.
\end{equation}
The solution is $S_t(1)=0$. Because the series defining $S_t(1)$ may not converge, we take a sequence $\{z_n\}\to 1$. The limit of $S_t(z_n)$ converges to the ODE above, and we can get the conclusion. Then we deduce:
\begin{equation*}
\sum\limits_{k=0}^\infty  \psi_{t}^{(k)}=0,\quad i.e.,\quad \sum\limits_{k=0}^\infty d^k \phi_{t}^{(k)}=0.
\end{equation*}

\subsection{Stationary Solution of  \eqref{riccati_tree} in Section \ref{section-tree-model}}\label{moment_generating_method_tree}

Define $R_t(z)=\sum_{k=0}^\infty z^k\ \phi_{t}^{(k)}$ where $0\leq z< 1$ and $\phi_{t}^{(k)}=\phi_{t}^{k}$ in equation (\ref{riccati_tree}) to avoid confusion. Without loss of generality, we assume $\e=1$. Then $R_{T}(z) = c ( 1 - d^{-1} z) $ and for $0 \le t \le T $
\begin{equation} \label{sum_tree}
\begin{split} 
\Dot{R}_t(z) &= \sum\limits_{k=0}^\infty z^k\Dot{\phi}_{t}^{(k)}
= \sum\limits_{k=0}^\infty z^k \sum\limits_{j=0}^k \phi_{t}^{(j)}\phi_{t}^{(k-j)}-1+\frac{z}{d}
=(R_t(z))^2- \Big(1-\frac{z}{d} \Big).  
\end{split} 
\end{equation}

By taking $T\to \infty$, the constant solution of equation (\ref{sum_tree}) satisfying $\Dot{R}_t(z)=0$ is $R(z)=\sqrt{1-\frac{z}{d}}$. We can then find constant solutions for $\phi$ functions by taking Taylor expansion and comparing it with $R(z)=\sum_{k=0}^\infty z^k\ \phi^{(k)}$, because 
\begin{equation*}
\begin{split} 
R(z)&=\sqrt{1-\frac{z}{d}}=\sum\limits_{k=0}^\infty  \binom{\frac{1}{2}}{k} \Big(-\frac{z}{d}\Big)^k 
=1-\frac{1}{2d} z-\sum\limits_{k=2}^\infty \frac{(2k-3)!! }{(2d)^k k!} z^k. \\
\end{split} 
\end{equation*}

\subsection{Solution $\overline{X}_t^1$ in  \eqref{evoution_eqn} in Section \ref{section-tree-model} } 
\label{solution_X^1bar}
First, according to proposition \ref{sol_markovc}, the formula for $\overline{X}_t^1$ is: 
\[
\begin{split}
\overline{X}_t^1&\, =\, \sum\limits_{j=1}^{\infty}\int^{t}_{0}  \frac{\,(t-s)^{2(j-1)}  \,}{\, (j-1)!\,} \cdot \rho_{j-1}(- (t-s)^{2}) \, e^{-(t-s)} \cdot  {\mathrm d} \overline{W}_{s}^{j}\\
&\, =\, \sum\limits_{k=0}^{\infty}\int^{t}_{0}  \frac{\,(t-s)^{2k}  \,}{\, k!\,} \cdot \rho_{k}(- (t-s)^{2}) \, e^{-(t-s)} \cdot  {\mathrm d} \overline{W}_{s}^{k+1}.
\end{split}
\]
By definition, $(\overline{W}_t^{k})_{0\leq t\leq T},k\geq 1$ are independent Brownian motions and the variance of $\overline{W}_{s}^{k+1}$ is $\dfrac{s}{d^k}$. Then similar to Appendix \ref{appendix3} and \ref{appendixvariance}, we have
\[
\text{Var} ( \overline{X}_t^1)\, =\,  \sum\limits_{k=0}^{\infty} \int^{t}_{0} \frac{\,(t-s)^{4k}\,}{\,(k!)^{2} \,} \lvert  \rho_{k} (- (t-s)^{2}) \rvert^{2} e^{-2(t-s)}\cdot\frac{1}{d^k}{\mathrm d} s 
\]
\[
{\, =\, \sum_{k=1}^{\infty}  \int^{t}_{0} \frac{\,2\,}{\,\pi\,}  \frac{\,\nu^{2k+1}\,}{\,(k!)^{2}\, 4^{k}\,} \big( K _{k-(1/2)}(\nu) \big)^{2} \cdot\frac{1}{d^k}{\mathrm d} \nu + \frac{\,1 - e^{-2t}\,}{\,2\,} \, ; \quad t \ge 0 \, . } 
\]
And 
\begin{equation*}
\begin{split}
\lim_{t\to \infty} \text{Var} ( \overline{X}_t^1) \, &=\,  \frac{\,1\,}{\,2\,} + \sum_{k=1}^{\infty} \int^{\infty}_{0} \frac{\,2\,  s^{2k+1}\,}{\,\pi (k!)^{2} 4^{k}\,} \cdot [ K_{k-(1/2)}(s)]^{2} \cdot\frac{1}{d^k}{\mathrm d} s 
\\
\,& =\, \frac{\,1\,}{\,2\,} + \frac{1}{2} \sum_{k=1}^{\infty}   {4k \choose 2k} \frac{1}{2k+1} \frac{1}{2^{4k}d^k} \,=\, \frac{1}{2} \sum_{k=0}^{\infty}   {4k \choose 2k} \frac{1}{2k+1} \frac{1}{(2d^{1/4})^{4k}} \,\\
\, &=\,  \frac{\,1\,}{\,2\,}\, \cdot \, \frac{\sqrt{2}}{4}\,2d^{1/4}\,\sqrt{4d^{1/2}-\sqrt{16d-16}}\\
 \, &=\,  \frac{\,\sqrt{2}\,}{\,2\,}\,d^{1/4}\sqrt{\sqrt{d}-\sqrt{d-1}}
=\,  \frac{\,\sqrt{2}\,}{\,2\,}\,\dfrac{d^{1/4}}{\sqrt{\sqrt{d}+\sqrt{d-1}}}
 \\
  \, &=\,\frac{\,\sqrt{2}\,}{\,2\,}\, \Big( 1 + \sqrt{ \frac{\,d-1\,}{\,d\,}}\Big)^{-1/2}
  \in \Big (\frac{1}{2}, \frac{\,\sqrt{2}\,}{\,2\,} \Big].
\end{split}
\end{equation*} 
The limit is monotone in $d$, it achieves maximum of $\frac{\,\sqrt{2}\,}{\,2\,}$ when $d=1$.

\newpage
\bibliographystyle{acm.bst}
\bibliography{references}

\begin{thebibliography}{1}

\bibitem{CarmonaFouqueSunSystemicRisk}
{\sc {Carmona}, R., {Fouque}, J.-P., and {Sun}, L.-H.}
\newblock {Mean Field Games and Systemic Risk}.
\newblock {\em Communications in Mathematical Sciences 13}, 4 (2015), 911--933.

\bibitem{delarue:hal-01457409}
{\sc Delarue, F.}
\newblock {Mean Field Games: A Toy Model On An Erdos-Renyi Graph}.
\newblock In {\em {Journ{\'e}es MAS 2016 de la SMAI -- Ph{\'e}nom{\`e}nes
  complexes et h{\'e}t{\'e}rog{\`e}nes.}\/} (Grenoble, France, 2017), vol.~60
  of {\em ESAIM: Procs}.

\bibitem{Nils-JP-Ichiba2018DirectedChain}
{\sc {Detering}, N., {Fouque}, J.-P., and {Ichiba}, T.}
\newblock {Directed Chain Stochastic Differential Equations}.
\newblock {\em Stochastic Processes and Their Applications 130}, 4 (2020),
  2519--2551.

\bibitem{ConvNashtoMFGlimit}
{\sc {Lacker}, D.}
\newblock {On the convergence of closed-loop Nash equilibria to the mean field
  game limit}.
\newblock {\em arXiv e-prints\/} (Aug 2018), arXiv:1808.02745.

\bibitem{LargeNetworkof_InteractingDiffusions}
{\sc {Lacker}, D., {Ramanan}, K., and {Wu}, R.}
\newblock {Large sparse networks of interacting diffusions}.
\newblock {\em arXiv e-prints\/} (Apr 2019), arXiv:1904.02585.

\bibitem{VanghanIEEE69}
{\sc {Vaughan}, D.}
\newblock A negative exponential solution for the matrix riccati equation.
\newblock {\em IEEE Transactions on Automatic Control 14}, 1 (February 1969),
  72--75.

\end{thebibliography}

\end{document}